\documentclass[11pt]{article}
\usepackage{amsmath}
\usepackage{amsfonts}
\usepackage{appendix}
\usepackage{amsthm}
\usepackage{amssymb}
\usepackage{enumitem}
\usepackage{yhmath}
\usepackage{fullpage}
\usepackage[all]{xy}
\usepackage{marginnote}
\usepackage{pgf,tikz,pgfplots}
\pgfplotsset{compat=1.15}
\usepackage{mathrsfs}
\usepackage{color}
\usetikzlibrary{arrows}
\numberwithin{equation}{section}
\usepackage{graphicx}
\usepackage{caption}
\usepackage{subcaption}
\usepackage[english]{babel}
\usepackage{tocloft}
\usepackage[colorlinks, linkcolor=blue, citecolor=blue, urlcolor=blue]{hyperref}
\usepackage{latexsym}
\usepackage{mathrsfs}
\usepackage{pifont}
\usepackage{authblk}

\newcommand{\R}{\mathbb{R}}

\newcommand{\la}{\lambda}

\newtheoremstyle{rem}{3pt}{3pt}{}{}
{\bfseries}{.}{.5em}{}

\newtheorem{theo}{Theorem}[section]
\newtheorem*{theo*}{Theorem}
\newtheorem{lemm}[theo]{Lemma}

\newtheorem{prop}[theo]{Proposition}

\newtheorem*{coro*}{Corollary}

\theoremstyle{rem}
\newtheorem{remax}[theo]{Remark}

\newenvironment{rema}
  {\pushQED{\qed}\remax}
  {\popQED\endremax}

\newtheorem{examx}[theo]{Example}

\newenvironment{exam}
  {\pushQED{\qed}\examx}
  {\popQED\endexamx}

\theoremstyle{definition}
\newtheorem{defi}[theo]{Definition}

\newtheorem*{term*}{Notation/Terminology}

\usepackage{geometry}
\title{Realisations of Racah algebras using Jacobi operators and convolution identities}

\author{Q. Labriet\footnote{Department of Mathematics, Aarhus University, Ny Munkegade 118, DK-8000, Aarhus C, Denmark. \emph{email adress}: quentin.labriet@math.au.dk}, L. Poulain d'Andecy\footnote{Laboratoire de math\'ematiques de Reims LMR, UMR 9008, Universit\'e de Reims Champagne-Ardenne, Moulin de la Housse BP 1039, 51100 Reims, France. \emph{email adress}: loic.poulain-dandecy@univ-reims.fr}}

\date{\today}
\begin{document}
\maketitle

\begin{abstract}
Using the representation theory of $sl_2$ and an appropriate model for tensor product of lowest weight Verma modules, we give a realisation first of the Hahn algebra, and then of the Racah algebra, using Jacobi differential operators. While doing so we recover some known convolution formulas for Jacobi polynomials involving Hahn and Racah polynomials. Similarly, we produce realisations of the higher rank Racah algebras in which maximal commutative subalgebras are realised in terms of Jacobi differential operators.
\end{abstract}

\section{Introduction}

The Racah algebra is associated to the Racah polynomials, the family of orthogonal polynomials sitting on top of the Askey scheme of orthogonal polynomials in one variable. The Racah algebra reflects in some sense the bispectral properties (recurrence and difference equations) of the Racah polynomials. This algebra has appeared in many places and has found applications in various fields of physics and mathematics. We refer to \cite{CFR23,CGPV22,GVZ14,GVZ14b} and the references therein for more background on the Racah algebra.

In this paper, we emphasize and we use the following simple representation theoretic interpretation of the Racah algebra. The Racah algebra encodes the relations between two operators called $X$ and $Y$. Now a simple exercise in representation theory shows, roughly speaking, that when we have a representation of the Racah algebra where $X$ is diagonalisable with a certain non-degenerate spectrum, then $X$ and $Y$ together form what is called a Leonard pair \cite{Ter01}. Namely, when $X$ is diagonal, $Y$ is of a tridiagonal form and vice versa. The tridiagonality reflects the bispectral properties of the Racah polynomials, and what is most important for us, the Racah polynomials (more precisely, their values at certain points) then appear as transition coefficients between eigenbases of $X$ and eigenbases of $Y$. To summarise, in the representations of the Racah algebra are hidden change of basis identities involving the Racah polynomials. The exact same story is also valid for the Hahn algebra, which is the algebra associated to the Hahn polynomials.

This is the first goal of this note to look at some known identities involving the Hahn and the Racah polynomials from this point of view. More precisely, the identities under scrutiny here involve Jacobi polynomials. The one involving the Hahn polynomials relates the Jacobi polynomials and the usual monomial basis, while the one involving the Racah polynomials relates two different products of Jacobi polynomials. Due to their particular form, they are often called convolution identities; they look as follows:
\[(x+y)^NP^{\la_1,\la_2}_l(\frac{y-x}{x+y})=\sum_{k=0}^N W_{l,k}x^ky^{N-k}\ ,\]
and
\begin{multline*}
(x+y+z)^{N-l}(x+y)^lP_{N-l}^{\lambda_3,\lambda_1+\lambda_2+2l}(\frac{x+y-z}{x+y+z})P_l^{\lambda_2,\lambda_1}(\frac{x-y}{x+y})\\
=\sum_{k=0}^N U_{l,k}(x+y+z)^{N-k}(y+z)^kP_{N-k}^{\lambda_1,\lambda_2+\lambda_3+2k}(\frac{y+z-x}{x+y+z})P_k^{\lambda_2,\lambda_3}(\frac{z-y}{y+z})\ ,
\end{multline*}
where $P^{\alpha,\beta}$ are the Jacobi polynomials and the coefficients $W_{l,k}$ (respectively, $U_{l,k}$) are expressed in terms of Hahn polynomials (respectively, of Racah polynomials). The precise identities are given in the text in Theorem \ref{theo:FormuleConvolutionExplicite-Hahn} and Theorem \ref{theo:ConvolutionRacah}. These identities are well-known \cite{Du81,KVJ98,VdJ97,Xu15} and our goal in this note is to exhibit the explicit representations of, respectively, the Hahn and the Racah algebras which imply these identities according to the philosophy sketched above.

The interpretation of these identities given in \cite{KVJ98,VdJ97} in terms of the representation theory of Lie algebras is closely related to our approach. In fact the authors of \cite{KVJ98,VdJ97} interpret the various polynomials appearing in the identities above as vectors in a certain model of a tensor product of two or three Verma modules of $sl_2$ (or $su_{1,1}$), and they use the known expressions of the Clebsch--Gordan and Racah coefficients in terms of Hahn and Racah polynomials to derive the identities. What we wish to add into the picture is an explicit description of the operators which have the above polynomials as eigenvectors. 

Naturally, these operators will be expressed in terms of the hypergeometric operator related to the Jacobi polynomials. It is interesting to see the explicit combination of hypergeometric operators having the specific product of Jacobi polynomials above as eigenvectors. By construction, these operators will form representations of the Hahn and Racah algebras. The convolution identities then follow from the general simple algebraic facts discussed above about these algebras. Thus we see the realisations of the Hahn and Racah algebras in terms of Jacobi operators as the algebraic statement above the convolution identities. This provides an alternative approach to these identities. One advantage of this approach is that, being purely algebraic, no analytic continuation is needed to treat complex parameters.

\vskip .2cm
Having produced a realisation of the usual Racah algebra in terms of Jacobi operators, we now proceed to the similar study of the so-called higher-rank Racah algebras. These algebras are to the $n$-fold tensor products of $sl_2$-representations what the usual Racah algebra is to the case $n=3$, and are the subject of active recent studies (see for example \cite{CFR23,CGPV22,DBGvVV17,DBvV20}). The higher-rank Racah algebras are related to the multivariate Racah polynomials which appear as transition coefficients between common eigenbases of different commutative subalgebras \cite{CFR23,DBvV20,Pos15}. What we do is to provide several explicit realisations of the higher-rank Racah algebra in each of which a given commutative subalgebra is entirely expressed in terms of Jacobi operators. This allows to produce explicitly the eigenvectors for each commutative subalgebra as some products of Jacobi polynomials. Thus, again, we see these algebraic constructions as explanations of the appearances of these polynomials, which were constructed and interpreted in \cite{LVJ02,VdJ03} in terms of $sl_2$-representation theory. Furthermore this gives also an algebraic approach to the identities relating these products of Jacobi polynomials in terms of multivariate Racah polynomials. Note that the products of Jacobi polynomials also appeared as orthogonal polynomials on the simplex \cite{DX01,Labriet22}.  

We hope that the constructions related to the higher-rank Racah algebras in the present paper will be useful for the study of these algebras, their commutative subalgebras and the associated multivariate polynomials (as in \cite{CFR23}), which is in current active development. As a natural continuation of this work, we mention the other convolution identities from \cite{KVJ98} where the Jacobi polynomials are replaced by other polynomials from the Askey scheme. The related algebraic realisations of the Hahn and Racah algebras, in the spirit of the present paper, are to be found. Finally, of course, the $q$-version of this work should be investigated involving the quantum group $U_q(sl_2)$ and the Askey--Wilson algebra \cite{CFGPRV21} instead of the Racah algebra.

\paragraph{Acknowledgements.} The authors warmly thank N. Cramp\'e and E. Koelink for their interest in this work. The first author is supported by a research grant from the Villum Foundation (Grant No. 00025373). The second author is supported by Agence Nationale de la Recherche Projet AHA ANR-18-CE40-0001 and the international research project AAPT of the CNRS.

\section{Notations for Jacobi polynomials and Jacobi operators}\label{secnotations}

All along the paper, we follow \cite{KLS10}. We indicate when we adapt slightly the notations for our purpose.

\paragraph{Notations.} The Pochhammer symbols are, for any quantity $x$ and any non-negative integer $N$:
\[(x)_N:=\prod_{i=0}^{N-1}(x+i)=x(x+1)\dots(x+N-1)\ .\]
The hypergeometric series are
\[{}_kF_l\left(\begin{array}{c} a_1 \,,\ \dots\ ,\ a_k \\ b_1\,,\ \dots\ ,\ b_l\end{array}; x\right)=\sum_{n\geq 0}\frac{(a_1)_n\dots (a_k)_n}{(b_1)_n\dots (b_l)_n}\frac{x^n}{n!}\ .\]

\paragraph{Jacobi polynomials.} Let $\la,\la'\in\mathbb{C}$. The Jacobi polynomials are:
\[P^{\la,\la'}_l(x)=\frac{(\la)_l}{l!}{}_2F_1\left(\begin{array}{c} -l \,,\ l+\la+\la'-1 \\ \la\end{array}; \frac{1-x}{2}\right)\ .\]
The more classical parametrisation, as in \cite{KLS10}, is with $\alpha=\la-1$ and $\beta=\la'-1$. The Jacobi polynomials are also given by:
\[
P^{\la,\la'}_l(x)=\frac{1}{2^l}\sum_{s=0}^l(-1)^s\binom{l+\la-1}{l-s}\binom{l+\la'-1}{s}(1-x)^s(1+x)^{l-s}\ .
\]
The Jacobi polynomials satisfy the differential equation:
\[l(l+\la+\la'-1)P^{\la,\la'}_l(x)=\Phi^{\la,\la'}_{x}P_l(x)\,,\ \ \ \ l\geq0\ .\]
where $\Phi^{\la,\la'}_{x}$ is the Jacobi operator:
\begin{equation}\label{Jacobioperator}
\Phi^{\la,\la'}_x=(x^2-1)\partial^2_x+\bigl(\la-\la'+(\la+\la')x\bigr)\partial_x\ ,
\end{equation}
where $\partial^k_x$ is the $k$-th derivative for the variable $x$.

\section{The Hahn algebra and Jacobi polynomials}

In all this section, we fix a positive integer $N$ and $\la_1,\la_2\in\mathbb{C}$ with the following conditions:
\begin{equation}\label{cond-Hahn}
\la_1,\la_2\notin\{0,-1,\dots,-(N-1)\}\ \ \ \text{and}\ \ \ \la_1+\la_2\notin\{0,-1,\dots,-(2N-2)\}\ .
\end{equation}

\subsection{The Hahn polynomials and the Hahn algebra}

\paragraph{Hahn polynomials.}
The Hahn polynomials are:
\[Q_k(x;\la_1,\la_2,N)={}_3F_2\left(\begin{array}{c} -k \,,\ k+\la_1+\la_2-1\,,\ -x \\ \la_1\,,\ -N\end{array}; 1\right)\ \ \ \quad \text{for $k=0,\dots,N$}.\]
More usual parameters, as in \cite{KLS10}, are $\alpha=\la_1-1$ and $\beta=\la_2-1$. The parameters being fixed, we will just write $Q_k(x)$. The Hahn polynomials satisfy the difference relations:
\[k(k+\la_1+\la_2-1)Q_k(x)=B(x)Q_k(x+1)+M(x)Q_k(x)+D(x)Q_k(x-1)\,,\ \ \ \ k=0,\dots,N\ ,\]
where
\[B(x)=(x-N)(x+\la_1),\ \ \ \ D(x)=x(x-\la_2-N),\ \ \ \ M(x)=-B(x)-D(x)\ .\]
They satisfy also the recurrence relations:
\[xQ_k(x)=A_kQ_{k+1}(x)+N_kQ_k(x)+C_kQ_{k-1}(x)\,,\ \ \ \ k=0,\dots,N\ ,\]
where
\[\left\{\begin{array}{l}\displaystyle A_k=\frac{(k-N)(k+\la_1)(k+\la_1+\la_2-1)}{(2k+\la_1+\la_2-1)(2k+\la_1+\la_2)}\,,\\[1em]
\displaystyle C_k=-\frac{k(k+\la_2-1)(k+\la_1+\la_2+N-1)}{(2k+\la_1+\la_2-2)(2k+\la_1+\la_2-1)}\,,\end{array}\right.
\ \ \ \ N_k=-A_k-C_k\ .\]
The conditions (\ref{cond-Hahn}) ensure that the values $\{A_k,C_{k+1},B(k),D(k+1)\}_{k=0,\dots,N-1}$ are all well-defined and non-zero.

Define the following values and renormalised values of the Hahn polynomials, for $k,l=0,\dots,N$:
\[Q_{k,l}=Q_k(l)\ \ \ \ \ \text{and}\ \ \ \ \ \tilde{Q}_{k,l}=\frac{B(0)\dots B(l-1)}{D(1)\dots D(l)}\frac{A_0\dots A_{k-1}}{C_1\dots C_k}Q_{k,l}\ .\]
The formulas relating $\tilde{Q}_{k,l}$ and $Q_{k,l}$ are obtained by a direct calculation giving:
\begin{equation}\label{ABCD-Hahn}
\begin{array}{l}
\displaystyle\frac{B(0)\dots B(l-1)}{D(1)\dots D(l)}=\binom{N}{l}\frac{(\la_1)_l(\la_2)_{N-l}}{(\la_2)_N}\ ,\\[0.8em]
\displaystyle\frac{A_0\dots A_{k-1}}{C_1\dots C_k}=\binom{N}{k}\frac{(\la_1)_k(\la_1+\la_2)_N}{(\la_2)_k(\la_1+\la_2+k-1)_k(\la_1+\la_2+2k)_{N-k}}\ .
\end{array}\end{equation}

\paragraph{The Hahn algebra.} To construct the Hahn algebra (see for example \cite{VZ19,Zhe93}), we consider the following two operators on polynomials:
\[X_{Hahn}=x\ \ \ \ \text{and}\ \ \ \ Y_{Hahn}=B(x)T_++M(x)\text{Id}+D(x)T_-\,,\]
where $x$ simply denotes the multiplication by $x$, and $T_{\pm}f(x)=f(x\pm1)$. The Hahn algebra encodes the commutation relations between these two operators.
\begin{defi}\label{defHahn}
The Hahn algebra is generated by $X,Y$ satisfying:
\[\begin{array}{l}
[X,Y]=Z\,,\\[0.5em]
[X,Z]=2X^2+(\la_1-\la_2-2N)X+Y-\la_1N\ ,\\[0.5em]
[Y,Z]=-2\{X,Y\}-(\la_1+\la_2)(\la_1+\la_2-2)X-(\la_1-\la_2-2N)Y+\la_1(\la_1+\la_2-2)N\ .
\end{array}\]
\end{defi}
Note that the first relation defines $Z$ in terms of $X$ and $Y$. By construction, the Hahn algebra is realized by the operators coming from the Hahn polynomials. Indeed, a straightforward calculation shows that the following is a representation of the Hahn algebra
\[X\mapsto X_{Hahn}\ \ \ \ \text{and}\ \ \ \ \ Y\mapsto Y_{Hahn}\,.\]

\paragraph{The Hahn algebra and transition coefficients.} Let $X,Y$ be two linear operators on an $(N+1)$-dimensional vector space such that:
\[Sp(X)=\{0,1,\dots,N\}\ \ \ \ \text{and}\ \ \ \ \text{$X,Y$ satisfy the Hahn algebra.}\]
These assumptions are enough to deduce that the transition coefficients between eigenbases of $X$ and $Y$ are expressed in terms of Hahn polynomials. More precisely, we can use the relations of the Hahn algebra to deduce the following properties.
\begin{prop}\label{prop-Hahn}
We can find two bases $\{\tilde{v}_l\}_{l=0,\dots,N}$ and $\{\tilde{w}_l\}_{l=0,\dots,N}$ such that the actions of $X$ and $Y$ are given by:
\[X\tilde{v}_l=l\tilde{v}_l\ \ \ \text{and}\ \ \ Y\tilde{v}_l=B(l-1)\tilde{v}_{l-1}+M(l)\tilde{v}_l+D(l+1)\tilde{v}_{l+1}\ ,\]
\[Y\tilde{w}_l=l(l+\la_1+\la_2-1)\tilde{w}_l\ \ \ \text{and}\ \ \ X\tilde{w}_l=C_l\tilde{w}_{l-1}+N_l\tilde{w}_l+A_l\tilde{w}_{l+1}\ .\]
Moreover, the transition coefficients are given by:
\begin{equation}\label{transition-Hahn}
\tilde{w}_l=\sum_{k=0}^NQ_{l,k}\tilde{v}_k\ \ \ \text{and}\ \ \ \tilde{v}_l=\Gamma^{-1}\sum_{k=0}^N\tilde{Q}_{k,l}\tilde{w}_k\ ,
\end{equation}
where $\Gamma$ is a constant given by:
\begin{equation}\label{Gam-Hahn}\Gamma=\sum_{k=0}^N\frac{B(0)\dots B(k-1)}{D(1)\dots D(k)}=\sum_{k=0}^N\frac{A_0\dots A_{k-1}}{C_1\dots C_k}=\frac{(\la_1+\la_2)_N}{(\la_2)_N}\ .
\end{equation}
\end{prop}
\begin{proof}
$\bullet$ Take an eigenbasis $\{\tilde{v}_l\}_{l=0,\dots,N}$ of $X$ and write
\[Y\tilde{v}_l=\sum_{i\geq 0}y_{i,l}\tilde{v}_i\ .\]
From the commutation relation between $X$ and $Z$ in the Hahn algebra, one gets that $y_{i,l}=0$ if $i\notin\{l-1,l,l+1\}$, and one also gets that $y_{l,l}=M(l)$. Then, from the commutation relation between $Y$ and $Z$, one gets that:
\[2(y_{l+1,l}y_{l,l+1}-y_{l-1,l}y_{l,l-1})=-4ly_{l,l}-(\la_1+\la_2)(\la_1+\la_2-2)l-(\la_1-\la_2-2N)y_{l,l}+\la_1(\la_1+\la_2-2)N\ .\]
This gives, recursively on $l$, a unique solution for $y_{l+1,l}y_{l,l+1}$, which is checked to be $B(l)D(l+1)$. Note that this never vanishes for $l\in\{0,\dots,N-1\}$. Therefore we can always rescale the vectors $\tilde{v}_l$ (a diagonal change of basis) such that $y_{l+1,l}=D(l+1)$ and $y_{l,l+1}=B(l)$. This gives the desired action of $Y$.

$\bullet$ Then we define $\{\tilde{w}_l\}_{i=0,\dots,N}$ by the first formula in (\ref{transition-Hahn}) and we apply directly $Y$ and $X$ on it. For the action of $Y$, the statement follows from a direct application of the difference equation for the Hahn polynomials, using the fact that $B(N)=0$. For the action of $X$, the statement follows from a direct application of the recurrence relation for the Hahn polynomials.

$\bullet$ Next, we note that the basis $\{\tilde{v}_l\}_{i=0,\dots,N}$ is uniquely fixed up to a global factor by the requirements for the actions of $X$ and $Y$. This is a straightforward calculation to check that the action of $X$ and $Y$ on the vectors:
\[\tilde{v}'_l=\sum_{k=0}^N\tilde{Q}_{k,l}=\frac{B(0)\dots B(l-1)}{D(1)\dots D(l)}\sum_{k=0}^NQ_{k,l} \frac{A_0\dots A_{k-1}}{C_1\dots C_k}\tilde{w}_k\,,\]
gives the same action as in item 1. One uses the recurrence relation for the action of $X$ and the difference relation for the action of $Y$. We conclude that $\tilde{v}_l=\Gamma^{-1} \tilde{v}'_l$ for some global factor $\Gamma$.

To find a first expression for $\Gamma$, we write:
\[\tilde{w}_0=\sum_{k=0}^N\tilde{v}_{k}=\Gamma^{-1}\Bigl(\sum_{k=0}^N\frac{B(0)\dots B(k-1)}{D(1)\dots D(k)}\Bigr)\tilde{w}_0+...\text{ (terms with $\tilde{w}_i$, $i>0$) }...\ .\]
We conclude that:
\[\Gamma=\sum_{k=0}^N\frac{B(0)\dots B(k)}{D(1)\dots D(k)}=\sum_{k=0}^N\binom{N}{k}\frac{(\la_1)_k(\la_2)_{N-k}}{(\la_2)_N}=\frac{(\la_1+\la_2)_N}{(\la_2)_N}\ .\]
We used (\ref{ABCD-Hahn}) and the last equality is a classical combinatorial equality. Finally, for the second expression for $\Gamma$, we proceed as above starting now from $\tilde{v}_0$.
\end{proof}
\begin{rema}\label{rema-ortho-Hahn}
Applying the two relations in (\ref{transition-Hahn}) successively (in one order and then in the other), the orthogonality relations for the Hahn polynomials follow:
\[\sum_{k=0}^NQ_{k,l}\tilde{Q}_{k,l'}=\delta_{l,l'}\Gamma\ \ \ \ \ \text{and}\ \ \ \ \ \sum_{l=0}^NQ_{k,l}\tilde{Q}_{k',l}=\delta_{k,k'}\Gamma\]
for any $k,k',l,l'=0,\dots,N$:
\end{rema}
Note that the formula for $\Gamma$ involving the coefficients $A_i,C_i$ in (\ref{Gam-Hahn}), together with (\ref{ABCD-Hahn}), leads to a non-trivial equality:
\begin{equation}\label{sum-Gam-Hahn}
\sum_{k=0}^N\binom{N}{k}\frac{(\la_1)_k}{(\la_2)_k(\la_1+\la_2+k-1)_k(\la_1+\la_2+2k)_{N-k}}=\frac{1}{(\la_2)_N}\ .
\end{equation}


\subsection{Realisations of the Hahn algebra in $U(sl_2)\otimes U(sl_2)$}\label{subsec-Hahn-sl2}

Take $H,E,F$ the generators of the Lie algebra $sl_2$ satisfying:
\[[H,E]=2E\,,\ \ \ \ [H,F]=-2F\,,\ \ \ \ [E,F]=H\ .\]
The diagonal embedding of $U(sl_2)$ into $U(sl_2)^{\otimes 2}$ is the algebra morphism defined on the generators by 
\[\delta(x)=x\otimes 1+1\otimes x\ \ \ \ \text{for $x\in\{H,E,F\}$.}\]
The Casimir element generating the center of $U(sl_2)$ is:
\[C=\frac{H^2+2H}{4}+FE\ .\]
To identify efficiently the Hahn algebra and its parameters, we fix now $\la_1,\la_2\in\mathbb{C}$. Recall that in a lowest-weight module of weight $\la$, the Casimir element $C$ is equal to the scalar $\frac{1}{4}\la(\la-2)$. Thus we will work in the quotient of $U(sl_2)^{\otimes 2}$ where the following central elements are specialized accordingly:
\[C\otimes 1 =\frac{\la_1(\la_1-2)}{4}\ \ \ \ \text{and}\ \ \ \ 1\otimes C =\frac{\la_2(\la_2-2)}{4}\ .\]
Equivalently, this is the quotient of $U(sl_2)^{\otimes 2}$ by the kernel of the representation $V_{\la_1}\otimes V_{\la_2}$, where $V_{\la_i}$ is a Verma module of lowest-weight $\la_i$. In other words, we aim directly at working in a tensor product of lowest-weight modules with weights $\la_1$ and $\la_2$, such as $V_{\la_1}\otimes V_{\la_2}$.

Finally we define the following elements of $U(sl_2)^{\otimes 2}$:
\begin{equation}\label{abstract-realisation}
X=\frac{H\otimes 1-\la_1}{2}\,,\ \ \ \ \ Y=\delta(C)-\frac{(\la_1+\la_2)(\la_1+\la_2-2)}{4}\,,\ \ \ \ \ \ h_{12}=\frac{\delta(H)-(\la_1+\la_2)}{2}.
\end{equation}
The choices made here are in order to recover immediately and in the clearest possible way the Hahn algebra as formulated above.
Indeed, the Hahn algebra is found in the afore-mentioned specialization of $U(sl_2)^{\otimes 2}$. A straightforward calculation gives:
\begin{prop}
The subalgebra generated by $X,Y$ and $h_{12}$ realises the Hahn algebra of Definition \ref{defHahn} when $N$ is replaced by $h_{12}$.
\end{prop}
Note that here $h_{12}$ is central element (it commutes with $X$ and $Y$) while $N$ was treated as a parameter in the Hahn algebra. When we consider a representation, and this is what we will happen for us, where $h_{12}$ is diagonalisable with positive integer eigenvalues, and we restrict to the eigenspace for the eigenvalue $N$, then the genuine Hahn algebra of Definition \ref{defHahn} is realised and the discussion of the preceding subsection applies.


\subsection{Hahn algebra, Jacobi polynomials and convolution identities}

\paragraph{Realisation of Verma modules on polynomials.} The Verma module $V_{\lambda}$ of lowest-weight $\lambda$ is realised on the space of polynomials $Pol(x)$ in one variable $x$ by:
\begin{equation}\label{realisationVerma}
H\mapsto \la+2 x\partial_x\,,\quad\ \ \ \ E\mapsto x\,,\quad\ \ \ \ F\mapsto -x\partial_x^2-\la\partial_x\ .
\end{equation}
Note that we denote by $x$ the multiplication by $x$ and by $\lambda$ the multiplication by the scalar $\lambda$. We will always do similarly from now on.

Now moving on to the tensor product $V_{\la_1}\otimes V_{\la_2}$, it is realised on the space $Pol(x)\otimes Pol(y)=Pol(x,y)$. The algebra $U(sl_2)\otimes U(sl_2)$ acts on this space namely, the elements $H\otimes 1$, $E\otimes 1$ , $F\otimes 1$ act by the formulas above, while the elements $1\otimes H$, $1\otimes E$ , $1\otimes F$ act by the same formulas with $x$ replaced by $y$. Of course, we have:
\begin{equation}\label{eq:Modele(x,y)}
\begin{array}{lcl}
H\otimes 1+1\otimes H & \mapsto & \la_1+\la_2+2x\partial_x+2y\partial_y\ ,\\[0.2em]
E\otimes 1+1\otimes E & \mapsto & x+y\ ,\\[0.2em]
F\otimes 1+1\otimes F & \mapsto  & -x\partial_x^2-y\partial^2_y-\la_1\partial_x-\la_2\partial_y\ .
\end{array}
\end{equation}

\paragraph{Change of variables.} In this model, the polynomials $x^ny^m$ are eigenvectors for $H\otimes 1$ and $1\otimes H$. To find the eigenvectors for the diagonal Casimir $\delta(C)$ it is convenient to go to another model in which it will have a familiar expression. For this, we use the change of coordinates 
\begin{equation}\label{tv-Hahn}
(t,v)=(x+y,\frac{y-x}{x+y})\ \ \ \quad\Leftrightarrow\quad \ \ \ (x,y)=(\frac{t(1-v)}{2},\frac{t(1+v)}{2})\ .
\end{equation}
This change of variables has a geometrical interpretation in terms of symmetric cones when one consider the tensor product of two holomorphic discrete series representations of $SL_2(\R)$ (see \cite{Labriet22}). Note that in this model the representation space is the space of polynomials in $t$ and $tv$.

In this new model, the various elements for the diagonal embedding of $U(sl_2)$ act as follows, where we see appearing explicitly the Jacobi operator $\Phi^{\la_1,\la_2}_v$ from Section \ref{secnotations}.
\begin{prop}\label{propmodele(t,v)}
We have:
\begin{equation}\label{eq:ModeleVermaModule(t,v)}
\begin{array}{lcl}
H\otimes 1+1\otimes H & \mapsto & \la_1+\la_2+2t\partial_t\ ,\\[0.2em]
E\otimes 1+1\otimes E & \mapsto & t\ ,\\[0.2em]
F\otimes 1+1\otimes F & \mapsto  & -t\partial_t^2-(\la_1+\la_2)\partial_t+t^{-1}\Phi^{\la_1,\la_2}_v\ ,
\end{array}
\end{equation}
and
\begin{equation}\label{eq:C(t,v)}
\delta(C)\mapsto \frac{(\la_1+\la_2)(\la_1+\la_2-2)}{4}+\Phi^{\la_1,\la_2}_v\ .
\end{equation}
\end{prop}
\begin{proof}
The formulas \eqref{eq:ModeleVermaModule(t,v)} are directly obtained from \eqref{eq:Modele(x,y)} applying the change of variables $(x,y)\to(t,v)$ given in \eqref{tv-Hahn}. These are simple straightforward calculations, using:
\[\partial_x=\partial_t-\frac{1+v}{t}\partial_v\ \ \ \ \text{and}\ \ \ \ \partial_y=\partial_t+\frac{1-v}{t}\partial_v\ .\]
Then it is easy to deduce the formula for $\delta(C)$, since the parts in the preceding formulas containing only $t$ combine to give the scalar $\frac{(\la_1+\la_2)(\la_1+\la_2-2)}{4}$ (this is the classical calculation in one variable). The only additional term is obtained in $\delta(F)\delta(E)$ and gives immediately $\Phi^{\la_1,\la_2}_v$.
\end{proof}
Of course, the key point is that the Casimir operator is now an ordinary differential operator only with derivatives in $v$, and moreover one differential operator that we know and understand very well. 

\begin{rema}
The presence of $t^{-1}$ in the third formula in \eqref{eq:ModeleVermaModule(t,v)} is not a problem since we recall that the representation space is the space of polynomials in $t$ and in $tv$ and we note that the operator $\Phi^{\la_1,\la_2}_v$ is $0$ on the polynomials not containing $v$. So the formulas in the above Proposition indeed provide an action on the representation space, as they should since we started from a representation in the first place. Similar comments will apply for other formulas later in the paper, but we will not point them out anymore.
\end{rema}

\paragraph{Eigenvectors and convolution identities.} Recall that the elements realizing the Hahn algebra in $U(sl_2)^{\otimes 2}$ are:
\[
X=\frac{H\otimes 1-\la_1}{2}\,,\ \ \ \ \ Y=\delta(C)-\frac{(\la_1+\la_2)(\la_1+\la_2-2)}{4}\,,\ \ \ \ \ \ h_{12}=\frac{\delta(H)-(\la_1+\la_2)}{2}\ .
\]
Their action in the representation $V_{\la_1}\otimes V_{\la_2}$ is given by:
\begin{equation}\label{action-generators-Hahn}\begin{array}{lcl}
X & \mapsto & x\partial_x\,,\\[0.2em]
h_{12} & \mapsto & x\partial_x+y\partial_y=t\partial_t\,,\\[0.2em]
Y & \mapsto & \Phi^{\la_1,\la_2}_v\ .
\end{array}
\end{equation}
where we use either the coordinates $(x,y)$ or the coordinates $(t,v)$ (or both for $h_{12}$) related
by \eqref{tv-Hahn}, and we use the formulas in Proposition \ref{propmodele(t,v)}.

Consider the two families of vectors:
\[v_l=x^ly^{N-l}\ \ \ \text{and}\ \ \ \ \ w_l=t^NP^{\la_1,\la_2}_l(v)=(x+y)^NP^{\la_1,\la_2}_l(\frac{y-x}{x+y})\,,\ \ \ \ l=0,1,\dots,N\,.\] 
We see at once that both families of vectors above are eigenvectors of $h_{12}$ with eigenvalues $N$, and that they are also eigenvectors for the two operators $X$ and $Y$, namely:
\[Xv_l=lv_l\ \ \ \ \text{and}\ \ \ \ Yw_l=l(l+\la_1+\la_2-1)w_l\ .\]
From the general properties of the Hahn algebra discussed in Proposition \ref{prop-Hahn}, we obtain rather directly the following identities.
\begin{theo}\label{theo:FormuleConvolutionExplicite-Hahn}
We have, for $l=0,1,\dots,N$,
\[w_l=\frac{(\la_1)_l}{l!}\sum_{k=0}^N\binom{N}{k}Q_{l,k}v_k\ ,\]
\[v_l=\sum_{k=0}^N\binom{N}{k}\frac{k!(\la_1)_l(\la_2)_{N-l}}{(\la_2)_k(\la_1+\la_2+k-1)_k(\la_1+\la_2+2k)_{N-k}}Q_{k,l} w_k\ ,\]
where we recall that $Q_{k,l}={}_3F_2\left(\begin{array}{c} -k \,,\ k+\la_1+\la_2-1\,,\ -l \\ \la_1\,,\ -N\end{array}; 1\right)$.
\end{theo}
\begin{proof}
Since $Xv_l=lv_l$ for $l=0,\dots,N$, the results of Proposition \ref{prop-Hahn} applies. So each eigenvector $v_l$ of $X$ (respectively, each eigenvector $w_l$ of $Y$) must be proportional to the vector $\tilde{v}_l$ (respectively, $\tilde{w}_l$) obtained in Proposition \ref{prop-Hahn}. We set  $v_l=\alpha_l\tilde{v}_l$ and $w_l=\beta_l\tilde{w}_l$. The equalities (\ref{transition-Hahn}) involving the Hahn coefficients now read:
\[w_l=\beta_l\sum_{k=0}^NQ_{l,k}\alpha_k^{-1} v_k\ \ \ \text{and}\ \ \ v_l=\alpha_l\Gamma^{-1}\sum_{k=0}^N\beta_k^{-1}\tilde{Q}_{k,l}w_k\ .\]
and it remains to calculate the unknown coefficients $\alpha$'s and $\beta$'s appearing in this relation. First we take $l=0$, and using that $w_0=(x+y)^N$, we find that $\beta_0\alpha_k^{-1}=\binom{N}{k}$. Then we put $x=0$ (for $l$ arbitrary) and using that $w_l=P_l(1)$ in this case, we find $P_l(1)=\beta_l\beta_0^{-1}$. This gives the first equality using that $P_l(1)=\frac{(\la_1)_l}{l!}$.

For the second relations, we know already that $\alpha_l\beta_k^{-1}=\bigl(\binom{N}{l}\frac{(\la_1)_k}{k!}\bigr)^{-1}$. Then it is a matter of simple algebraic manipulations to find the claimed equality, using the explicit formula \eqref{Gam-Hahn} for $\Gamma$ and the explicit relation \eqref{ABCD-Hahn} expressing $\tilde{Q}_{k,l}$ in terms of $Q_{k,l}$.
\end{proof}

\begin{rema}\label{rema-tensorproduct2}
From the point of view of the representation theory of $sl_2$, for a given $l\geq 0$, the polynomials $\{t^{l+n}P_l^{\la_1,\la_2}(v)\}_{n\geq 0}$ form a basis of a summand of $V_{\la_1}\otimes V_{\la_2}$ isomorphic to $V_{\la_1+\la_2+2l}$. In fact, this proves the existence of such a summand, under the restrictions \eqref{cond-Hahn} on the parameters $\la_1,\la_2$. If these restrictions are satisfied for any positive integer $N$, we have realised explicitly the decomposition of the tensor product:
\[V_{\la_1}\otimes V_{\la_2}=\bigoplus_{l\geq 0} V_{\la_1+\la_2+2l}\ .\]
In each summand $V_{\la_1+\la_2+2l}$, the lowest-weight vector is $t^l P_l^{\la_1,\la_2}(v)$.
\end{rema}

\subsection{Hahn algebra in the Jacobi algebra}

From the above, we can realize the Hahn algebra using only operators acting on the variable $v$. For this note that we have:
\[H\otimes 1\mapsto 2x\partial_x+\la_2=t\partial_t(1-v)-(1-v)(1+v)\partial_v+\la_2\ .\]
We get the following result.
\begin{prop}
The following is a representation of the Hahn algebra:
\[X\mapsto \frac{1}{2}N(1-v)-\frac{(1-v)(1+v)}{2}\partial_v\ \ \ \ \text{and}\ \ \ \ Y\mapsto \Phi^{\la_1,\la_2}_v\ .\]
\end{prop}
\begin{proof}
This follows immediately from what we have obtained so far, noting that the parameter $N$ (the image of the central element $h_{12}$ in the abstract realization) is here given by the operator $t\partial_t$. 
\end{proof}
\begin{rema}
The proposition above can also be checked directly without reference to the representation theory of $sl_2$ and to the realization of the Hahn algebra in $U(sl_2)^{\otimes 2}$. This calculation would be done in what is sometimes called the Jacobi algebra (generated by $v$ and $\Phi^{\la_1,\la_2}_v$). We refer to \cite{GIVZ16} where such a calculation is done.
\end{rema}

\section{The Racah algebra and Jacobi polynomials}

In this section, we fix a positive integer $N$ and $\la_1,\la_2,\la_3\in\mathbb{C}$ with the following conditions:
\begin{equation}\label{cond-Racah}
\la_1,\la_2,\la_3\notin\{0,-1,\dots,-(N-1)\}\ \ \ \text{and}\ \ \ \ \la_1+\la_2,\la_2+\la_3,\la_1+\la_2+\la_3\notin\{0,-1,\dots,-(2N-2)\}\ .
\end{equation}

\subsection{The Racah algebra}

\paragraph{The Racah polynomials.} The Racah polynomials are:
\[R_k(\mu(x))={}_4F_3\left(\begin{array}{c} -k\ ,\ k+\la_2+\la_3-1\ ,\ -x\ ,\ x+\la_1+\la_2-1 \\ \la_2\ ,\ \la_1+\la_2+\la_3+N-1\ ,\ -N\end{array}; 1\right)\ \ \ \ \text{for $k=0,\dots,N$},\]
where $\mu(x)=x(x+\la_1+\la_2-1)$. The parameters being fixed, we do not indicate them in the notation. The parameters are related with the more usual ones $\alpha,\beta,\gamma,\delta$ used in \cite{KLS10} by:
\[\alpha=\lambda_2-1\,,\ \ \ \ \beta=\lambda_3-1\,,\ \ \ \ \delta=\lambda_1+\lambda_2+N-1\,,\ \ \ \ \ \gamma=-N-1\ .\]
The Racah polynomials satisfy the difference equations:
\[k(k+\la_2+\la_3-1)R_k(\mu(x))=B(x)R_k(\mu(x+1))+M(x)R_k(\mu(x))+D(x)R_k(\mu(x-1))\,,\]
where
\[\begin{array}{l}\displaystyle B(x)=\frac{(x-N)(x+\la_2)(x+\la_1+\la_2-1)(x+\la_1+\la_2+\la_3+N-1)}{(2x+\la_1+\la_2-1)(2x+\la_1+\la_2)}\,,\\[1em]
\displaystyle D(x)=\frac{x(x+\la_1-1)(x-\la_3-N)(x+\la_1+\la_2+N-1)}{(2x+\la_1+\la_2-2)(2x+\la_1+\la_2-1)}\,,\end{array}\ \ \ \ \ \ M(x)=-B(x)-D(x)\ .\]
They satisfy also the recurrence relations:
\[x(x+\la_1+\la_2-1)R_k(\mu(x))=A_kR_{k+1}(\mu(x))+N_kR_k(\mu(x))+C_kR_{k-1}(\mu(x))\,,\]
where
\[A_k=B(k)|_{\la_1\leftrightarrow \la_3}\,,\ \ \ \ \\ C_k=D(k)|_{\la_1\leftrightarrow \la_3}\,,
\ \ \ \ N_k=-A_k-C_k\ .\]
The conditions (\ref{cond-Racah}) ensure that the values $\{A_k,C_{k+1},B(k),D(k+1)\}_{k=0,\dots,N-1}$ are all well-defined and non-zero.

Define the following values and renormalised values of the Racah polynomials:
\begin{equation}\label{Rkl}
R_{k,l}=R_k(\lambda(l))\ \ \ \text{and}\ \ \ \tilde{R}_{k,l}=\frac{B(0)\dots B(l-1)}{D(1)\dots D(l)}\frac{A_0\dots A_{k-1}}{C_1\dots C_k}R_{k,l}\ \ \ \ \ \forall k,l=0,\dots,N\ .
\end{equation}
The formulas relating $\tilde{R}_{k,l}$ and $R_{k,l}$ are obtained by a direct calculation giving:
\begin{equation}\label{ABCD-Racah}
\begin{array}{l}
\displaystyle\frac{B(0)\dots B(l-1)}{D(1)\dots D(l)}=\binom{N}{l}\frac{(\la_2)_l(\la_1+\la_2+\la_3+N-1)_{l}(\la_1+\la_2)_N}{(\la_1)_l(\la_3+N-l)_l(\la_1+\la_2+l-1)_l(\la_1+\la_2+2l)_{N-l}}\ ,\\[1em]
\displaystyle\frac{A_0\dots A_{k-1}}{C_1\dots C_k}=\displaystyle\frac{B(0)\dots B(k-1)}{D(1)\dots D(k)}|_{\la_1\leftrightarrow \la_3}\ .
\end{array}\end{equation}

\paragraph{The Racah algebra.} To construct the Racah algebra, we consider the following two operators on polynomials:
\[X_{Racah}=x(x+\la_1+\la_2-1)\ \ \ \ \text{and}\ \ \ \ Y_{Racah}=B(x)T_++M(x)\text{Id}+D(x)T_-\,,\]
where $x$ simply denotes the multiplication by $x$, and $T_{\pm}f(x)=f(x\pm1)$. The Racah algebra encodes the commutation relations between these two operators \cite{GVZ14,GVZ14b}.
\begin{defi}\label{defRacah}
The Racah algebra is generated by $X,Y$ satisfying:
\[\begin{array}{l}
[X,Y]=Z\,,\\[0.5em]
[X,Z]=2X^2+2\{X,Y\}+a_1X+a_2Y+a_3\ ,\\[0.5em]
[Y,Z]=-2Y^2-2\{X,Y\}-a_1Y-b_2X-b_3\ ,
\end{array}\]
where
\[\begin{array}{l} a_1=(\la_2 - 2N) (\la_1 +\la_2 +\la_3 + N - 1) - \la_2 (N + 1) - \la_1 \la_3\,,\\[1em]
a_2=(\lambda_1+\lambda_2)(\lambda_1+\lambda_2-2)\,,\ \ \ \ a_3=-\lambda_2 N (\lambda_1 + \lambda_2 + \lambda_3 + N-1)(\lambda_1 + \lambda_2-2) \ ,\\[1em]
\displaystyle b_2=(\lambda_2+\lambda_3)(\lambda_2+\lambda_3-2)\,,\ \ \ \ \ \displaystyle b_3= -\lambda_2 N (\lambda_1 + \lambda_2 + \lambda_3 + N-1)(\lambda_2 + \lambda_3-2)\ .
\end{array}\]
\end{defi}
Note that the first relation defines $Z$ in terms of $X$ and $Y$. By construction, the Racah algebra is realized by the operators coming from the Racah polynomials, namely, a straightforward calculation shows that the following is a representation of the Racah algebra:
\[X\mapsto X_{Racah}\ \ \ \ \text{and}\ \ \ \ \ Y\mapsto Y_{Racah}\,.\]

\paragraph{The Racah algebra and transition coefficients.} Let $X,Y$ be two linear operators on an $(N+1)$-dimensional vector space such that:
\[Sp(X)=\{l(l+\la_1+\la_2-1)\}_{l=0,\dots,N}\ \ \ \ \text{and}\ \ \ \ \text{$X,Y$ satisfy the Racah algebra.}\]
These assumptions are enough to deduce that the transition coefficients between eigenbases of $X$ and $Y$ are expressed in terms of Racah polynomials. More precisely, we can use the relations of the Racah algebra to deduce the following properties.
\begin{prop}\label{prop-Racah}
We can find two bases $\{\tilde{v}_l\}_{i=0,\dots,N}$ and $\{\tilde{w}_l\}_{i=0,\dots,N}$ such that the actions of $X$ and $Y$ are given by:
\[X\tilde{v}_l=l(l+\la_1+\la_2-1)\tilde{v}_l\ \ \ \text{and}\ \ \ Y\tilde{v}_l=B(l-1)\tilde{v}_{l-1}+M(l)\tilde{v}_l+D(l+1)\tilde{v}_{l+1}\ ,\]
\[Y\tilde{w}_l=l(l+\la_2+\la_3-1)\tilde{w}_l\ \ \ \text{and}\ \ \ X\tilde{w}_l=C_l\tilde{w}_{l-1}+N_l\tilde{w}_l+A_l\tilde{w}_{l+1}\ .\]
Moreover, the transition coefficients are given by:
\begin{equation}\label{transition-Racah}
\tilde{w}_l=\sum_{k=0}^NR_{l,k}\tilde{v}_k\ \ \ \text{and}\ \ \ \tilde{v}_l=\Gamma^{-1}\sum_{k=0}^N\tilde{R}_{k,l}\tilde{w}_k\ ,
\end{equation}
where $\Gamma$ is a constant given by:
\[\Gamma=\sum_{k=0}^N\frac{B(0)\dots B(k-1)}{D(1)\dots D(k)}=\sum_{k=0}^N\frac{A_0\dots A_{k-1}}{C_1\dots C_k}=\frac{(\la_1+\la_2)_N(\la_2+\la_3)_N}{(\la_1)_N(\la_3)_N}\ .\]
\end{prop}
\begin{proof}
The proof follows the exact same lines as for Proposition \ref{prop-Hahn} for the Hahn algebra, up to the final equality for $\Gamma$. As in Remark \ref{rema-ortho-Hahn}, we can apply the two relations in (\ref{transition-Racah}) successively  and we deduce the orthogonality relations for the Racah polynomials:
\[\sum_{k=0}^NR_{k,l}\tilde{R}_{k,l'}=\delta_{l,l'}\Gamma\ \ \ \ \ \text{and}\ \ \ \ \ \sum_{l=0}^NR_{k,l}\tilde{R}_{k',l}=\delta_{k,k'}\Gamma\ .\]
The formula for $\Gamma$ can thus be found in \cite{KLS10}.
\end{proof}
Note that the formula for $\Gamma$ involving the coefficients $A_i,C_i$ together with (\ref{ABCD-Racah}), leads to a non-trivial equality, generalising the one in the Hahn case (\ref{sum-Gam-Hahn}):
\[\sum_{l=0}^N\binom{N}{l}\frac{(\la_2)_l(\la_1+\la_2+\la_3+N-1)_{l}}{(\la_1)_l(\la_3+N-l)_l(\la_1+\la_2+l-1)_l(\la_1+\la_2+2l)_{N-l}}=\frac{(\la_2+\la_3)_N}{(\la_1)_N(\la_3)_N}\ .\]

\begin{rema}
The two operators $X$ and $Y$ form a Leonard pair \cite{Ter01}, that is, each one of them is tridiagonal in an eigenbasis of the other. This was also the case in the Hahn situation in Proposition \ref{prop-Hahn}. More general results in the same spirit than Proposition \ref{prop-Racah} on the representations of the Racah algebra can be found \cite{GVZ14,GVZ14b}. We only stated the version that we needed, namely assuming from the beginning the knowledge of the spectrum of $X$.
\end{rema}

\subsection{Realisation of the Racah algebra in $U(sl_2)^{\otimes 3}$.}
We need various embeddings of the Casimir element $C$ in $U(sl_2)^{\otimes 3}$. We refer to the next section for the definition of more general intermediate Casimir elements in the $n$-fold tensor product of $U(sl_2)$, of which the following are particular cases:
\[C_{12}=\delta(C)\otimes 1\,,\ \ \ \ \ C_{23}=1\otimes\delta(C)\ \ \ \ \ \text{and}\ \ \ \ \ C_{123}=(\delta\otimes \text{Id})\circ\delta(C)\ .\]
Since $C\otimes 1$ is central in $U(sl_2)^{\otimes 2}$, it commutes in particular with $\delta(C)$, and applying $\delta\otimes \text{Id}$ we get that $C_{123}$ and $C_{12}$ commute. A similar reasoning shows that $C_{123}$ commutes also with $C_{23}$. 

To identify efficiently the Racah algebra, we follow the same lines as we have done for the Hahn algebra in the previous section, and we consider the quotient of $U(sl_2)^{\otimes 3}$ where the central elements are specialized:
\[C\otimes 1 \otimes 1=\frac{\lambda_1(\lambda_1-2)}{4}\,,\ \ \ \ 1\otimes C\otimes  1=\frac{\lambda_2(\lambda_2-2)}{4}\ \ \ \ \text{and}\ \ \ \ 1\otimes 1\otimes C =\frac{\lambda_3(\lambda_3-2)}{4}\ .\]
We define the following elements of $U(sl_2)^{\otimes 3}$:
\begin{equation}\label{abstract-realisation2}
\begin{array}{c}\displaystyle X=C_{12}-\frac{(\lambda_1+\lambda_2)(\lambda_1+\lambda_2-2)}{4}\,,\ \ \ \ \ Y=C_{23}-\frac{(\lambda_2+\lambda_3)(\lambda_2+\lambda_3-2)}{4}\,,\\[0.8em]
\displaystyle C'_{123}=C_{123}-\frac{(\lambda_1+\lambda_2+\la_3)(\lambda_1+\lambda_2+\la_3-2)}{4}\ .
\end{array}
\end{equation}
 The choices made here are in order to recover immediately and in the clearest possible way the Racah algebra as formulated above. 
Indeed, the Racah algebra is found in the afore-mentioned specialization of $U(sl_2)^{\otimes 3}$. A straightforward and classical calculation gives:
\begin{prop}\label{prop:realRacah}
The subalgebra generated by $X,Y$ and $C'_{123}$ realises the Racah algebra of Definition \ref{defRacah} when $C'_{123}$ is replaced as follows:
\begin{equation}\label{eq:CpvsN}
C'_{123}=N(\la_1+\la_2+\la_3+N-1)\ .
\end{equation}
\end{prop}
Note that $C'_{123}$ is a central element (it commutes with $X$ and $Y$) and so we may treat it as a parameter. The formula in the proposition above is thus seen as a reparametrisation. Of course, in the representations we will be considering, the element $C'_{123}$ will be diagonalisable with eigenvalues given by the formula above with $N$ taking positive integer values. In other words, the proposition above says that the genuine Racah algebra of Definition \ref{defRacah} is realised as acting in the eigenspace of $C'_{123}$ with eigenvalue given above for a positive integer $N$. Note that when $C'_{123}$ has eigenvalue as in the proposition, this corresponds to the following (more recognisable) value of the total Casimir $C_{123}$:
\[C_{123}=\frac{(\lambda_1+\lambda_2+\lambda_3+2N)(\lambda_1+\lambda_2+\lambda_3+2N-2)}{4}\ .\]

\subsection{Racah algebra, Jacobi polynomials and convolution identities}

\paragraph{Tensor product of three lowest-weight Verma modules of $sl_2$.} We keep the realisation of the Verma module $V_{\lambda}$ of lowest-weight $\lambda$ on polynomials as in the previous section in \eqref{realisationVerma}. Then the tensor product $V_{\la_1}\otimes V_{\la_2}\otimes V_{\la_3}$ is realised on the space $Pol(x,y,z)$. The algebra $U(sl_2)^{\otimes 3}$ acts on this space, each copy acting, respectively, on each variable $x,y,z$. Of course, the diagonal embedding of $sl_2$ into $U(sl_2)^{\otimes 3}$ acts as follows:
\begin{equation}\label{eq:Modele(x,y,z)}
\begin{array}{lcl}
H\otimes 1\otimes 1+1\otimes H\otimes 1+ 1\otimes 1\otimes H & \mapsto & \la_1+\la_2+\la_3+2x\partial_x+2y\partial_y+2z\partial_z\ ,\\[0.2em]
E\otimes 1\otimes 1\,+1\otimes E\otimes 1\,+ 1\otimes 1\otimes E & \mapsto & x+y+z\ ,\\[0.2em]
F\otimes 1\otimes 1\,+1\otimes F\otimes 1\,+ 1\otimes 1\otimes F & \mapsto  & -x\partial_x^2-y\partial^2_y-z\partial^2_z-\la_1\partial_x-\la_2\partial_y-\la_3\partial_z\ .
\end{array}
\end{equation}

\paragraph{Changes of variables.} As for the Hahn algebra, we need to change the model in order to find naturally the eigenvectors for the elements realising the Racah algebra. It turns out that we will need different changes of variables, depending on which elements we want to diagonalise. We use the following changes of coordinates 
\begin{equation}\label{tv-Racah1}
(t,u_1,u_2)=(x+y+z,\ \frac{y-x}{x+y},\ \frac{z-(x+y)}{x+y+z})\ .
\end{equation}
\begin{equation}\label{tv-Racah2}
(t,v_1,v_2)=(x+y+z,\ \frac{z-y}{y+z},\ \frac{(y+z)-x}{x+y+z})\ .
\end{equation}
Note that the space of polynomials in $x,y,z$ becomes, respectively, the space of polynomials in $t,tu_2,tu_1(1-u_2)$ and the space of polynomials in $t,tv_2,tv_1(1+v_2)$.

Thanks to these changes of variables, we can express every elements of the Racah algebra using various Jacobi operators $\Phi^{\la,\la'}_v$ from Section \ref{secnotations}.
\begin{prop}\label{prop:actionRacah}
Using the first set of variables $(t,u_1,u_2)$, we have:
\[X\mapsto \Phi^{\la_1,\la_2}_{u_1},\ \ \ \ C'_{123}\mapsto\Phi^{\la_1+\la_2,\la_3}_{u_2}+\frac{2}{1-u_2}\Phi_{u_1}^{\la_1,\la_2}\,.\]
Using the second set of variables $(t,v_1,v_2)$, we have:
\[Y\mapsto \Phi^{\la_2,\la_3}_{v_1},\ \ \ \ C'_{123}\mapsto\Phi^{\la_1,\la_2+\la_3}_{v_2}+\frac{2}{1+v_2}\Phi_{v_1}^{\la_2,\la_3}\,.\]
\end{prop}
\begin{proof}
Let us do the first change of variables in two steps. We start with the first step, which replaces $(x,y)$ by $t_1=x+y$ and $u_1=\frac{y-x}{x+y}$ (and leaves $z$ as it is). At this stage, we already know from the calculations made in two variables in Proposition \ref{propmodele(t,v)} that the operators $X$ becomes $\Phi^{\la_1,\la_2}_{u_1}$. Those calculations also show that the diagonal embedding of $sl_2$ into $U(sl_2)^{\otimes 3}$ now acts as follows:
\[\begin{array}{lcl}
H\otimes 1\otimes 1+1\otimes H\otimes 1+ 1\otimes 1\otimes H & \mapsto & \la_1+\la_2+\la_3+2t_1\partial_{t_1}+2z\partial_z\ ,\\[0.2em]
E\otimes 1\otimes 1\,+1\otimes E\otimes 1\,+ 1\otimes 1\otimes E & \mapsto & t_1+z\ ,\\[0.2em]
F\otimes 1\otimes 1\,+1\otimes F\otimes 1\,+ 1\otimes 1\otimes F & \mapsto  & -t_1\partial_{t_1}^2-(\la_1+\la_2)\partial_{t_1}+t_1^{-1}\Phi_{u_1}^{\la_1,\la_2}-z\partial^2_z-\la_3\partial_z\ .
\end{array}\]
Now we perform the second step of the change of variables, which replaces $(t_1,z)$ by $t=t_1+z$ and $u_2=\frac{z-t_1}{t_1+z}$ (and does not touch $u_1$). We see that the same two-variables calculation as before gives:
\[\begin{array}{lcl}
H\otimes 1\otimes 1+1\otimes H\otimes 1+ 1\otimes 1\otimes H & \mapsto & \la_1+\la_2+\la_3+2t\partial_{t}\ ,\\[0.2em]
E\otimes 1\otimes 1\,+1\otimes E\otimes 1\,+ 1\otimes 1\otimes E & \mapsto & t\ ,\\[0.2em]
F\otimes 1\otimes 1\,+1\otimes F\otimes 1\,+ 1\otimes 1\otimes F & \mapsto  & -t\partial_{t}^2-(\la_1+\la_2+\la_3)\partial_{t}+t_1^{-1}\Phi_{u_1}^{\la_1,\la_2}+t^{-1}\Phi_{u_2}^{\la_1+\la_2,\la_3}\ ,
\end{array}\]
where we have not modified the term $t_1^{-1}\Phi_{u_1}^{\la_1,\la_2}$ in the action of $F$. To really conclude the change of variables, we need to replace $t_1$ in this term using the new variables. The formula is $t_1=\frac{1}{2}t(1-u_2)$. From the above formulas for the action of $H,E,F$ it is immediate to calculate the image of $C_{123}$. Indeed the terms containing only $t$ give, using the well-known calculation in one variable, the scalar $\frac{1}{4}(\la_1+\la_2+\la_3)(\la_1+\la_2+\la_3-2)$. This is exactly what we substract from $C_{123}$ to get $C'_{123}$. The remaining terms (in the product of the image of $E$ and $F$) give the claimed formula for $C'_{123}$.

The reasoning is completely parallel for the other change of variables, starting with the change of variables which replaces $(y,z)$ by $t_1=y+z$ and $v_1=\frac{z-y}{y+z}$ (and does not touch $x$). We skip the details which are exactly the same as above. The only difference in the result is the sign in front of $v_2$, which comes from the fact that now we have $t_1=\frac{1}{2}t(1+v_2)$.
\end{proof}

\paragraph{Eigenvectors of the Racah algebra.} Now that we have the explicit form of the operators $X,Y$ and $C'_{123}$ in terms of Jacobi operators, we can look for their eigenvectors, which naturally, are going to be expressed in terms of Jacobi polynomials. Recall that we look for the common eigenvectors of $X$ and $C'_{123}$ on one hand, and the common eigenvectors of $Y$ and $C'_{123}$ on the other hand.

Using the new variables in \eqref{tv-Racah1} and \eqref{tv-Racah2}, we define the two families of polynomials:
\begin{equation}\label{eq:vwtuv}
\begin{array}{l}v_l=\displaystyle\frac{t^N}{2^l}(1-u_2)^lP_{N-l}^{\la_1+\la_2+2l,\la_3}(u_2)P_l^{\la_1,\la_2}(u_1)\,,\\[0.8em]
w_l=\displaystyle\frac{t^N}{2^l}(1+v_2)^lP_{N-l}^{\la_1,\la_2+\la_3+2l}(v_2)P_l^{\lambda_2,\lambda_3}(v_1)\,,\end{array}\ \ \ \ l=0,1,\dots,N\,.
\end{equation}
or in the original variables $x,y,z$:
\begin{equation}\label{eq:vwxyz}
\begin{array}{l}v_l=(x+y+z)^{N-l}(x+y)^lP_{N-l}^{\la_1+\la_2+2l,\la_3}(\frac{z-(x+y)}{x+y+z})P_l^{\la_1,\la_2}(\frac{y-x}{x+y})\,,\\[0.8em]
w_l=(x+y+z)^{N-l}(y+z)^lP_{N-l}^{\la_1,\la_2+\la_3+2l}(\frac{y+z-x}{x+y+z})P_l^{\lambda_2,\lambda_3}(\frac{z-y}{y+z})\,,\end{array}\ \ \ \ l=0,1,\dots,N\,.
\end{equation}
\begin{prop}\label{prop:eigenvectorsRacah3}
The vectors $v_l$ and $w_l$ are eigenvectors for the operators $X$ and $Y$ satisfying the Racah algebra:
\[Xv_l=l(l+\lambda_1+\lambda_2-1)v_l\ \ \ \ \text{and}\ \ \ \ Yw_l=l(l+\lambda_2+\lambda_3-1)w_l\ .\]
and both families $\{v_l\}_{l=0,\dots,N}$ and $\{w_l\}_{l=0,\dots,N}$ are eigenvectors of $C'_{123}$:
\[C'_{123}v_l=N(\la_1+\la_2+\la_3+N-1)v_l\ \ \ \ \text{and}\ \ \ \ C'_{123}w_l=N(\la_1+\la_2+\la_3+N-1)w_l\ .\]
\end{prop}
\begin{proof}
In view of Proposition \ref{prop:actionRacah}, in the new variables, the operator $X$ acts as $\Phi_{u_1}^{\la_1,\la_2}$. This shows at once that the vectors $v_l$ are eigenvectors for $X$ with the given eigenvalue.

Now to show that the vectors $v_l$ are also eigenvectors for $C'_{123}$, we look at the expression of $C'_{123}$ in terms of the new variables given in Proposition \ref{prop:actionRacah}. We need the following simple Lemma.
\begin{lemm}\label{lemm:Jacobioperator}
As operators on polynomials in $u$, we have for any $l\geq 0$:
\[\Phi_u^{\alpha,\beta}\circ (1-u)^l=(1-u)^l\circ \Bigl(\Phi_u^{\alpha+2l,\beta}+l\beta-l(l+\alpha-1)\frac{1+u}{1-u}\Bigr)\ .\]
\end{lemm}
\begin{proof}[Proof of the lemma]
This is a direct calculation using Leibniz rule. Starting from $\Phi_u^{\alpha,\beta}\circ (1-u)^l$ and moving the derivatives through $(1-u)^l$, one gets:
\[(1-u)^l\Phi_u^{\alpha,\beta}+(u^2-1)\Bigl(l(l-1)(1-u)^{l-2}-2l(1-u)^{l-1}\partial_u\Bigr)-l(1-u)^{l-1}\Bigl(\alpha-\beta+(\alpha+\beta)u\Bigr)\ .\]
The term with $\partial_u$ can enter the Jacobi operator since $\Phi_u^{\alpha,\beta}+2l(1+u)\partial_u=\Phi_u^{\alpha+2l,\beta}$. The remaining terms of degree 0 in $\partial_u$ organise as claimed.
\end{proof}
Now using this lemma, it is easy to apply the operator $\Phi_{u_2}^{\la_1+\la_2,\la_3}$ on $v_l$. We find:
\[\Phi_{u_2}^{\la_1+\la_2,\la_3}(v_l)=(N-l)(N+l+\la_1+\la_2+\la_3-1)v_l+l\la_3 v_l-l(l+\la_1+\la_2-1)\frac{1+u_2}{1-u_2}v_l\ .\]
The remaining part of the action of $C'_{123}$ gives immediately:
\[\frac{2}{1-u_2}\Phi_{u_1}^{\la_1,\la_2}(v_l)=l(l+\la_1+\la_2-1)\frac{2}{1-u_2}v_l\ .\]
Summing these two formulas, we find that only $N(N+\la_1+\la_2+\la_3-1)v_l$ remains, as claimed.

The proof of the formulas for the vectors $w_l$ is exactly similar, using $\Phi_{-u}^{\alpha,\beta}=\Phi_u^{\beta,\alpha}$.
\end{proof}

\paragraph{Convolution identities.} Now that we have identified the eigenvectors of $X$ and $Y$, we can apply the general properties of the Racah algebra discussed in Proposition \ref{prop-Racah} and deduce the following identities.
\begin{theo}\label{theo:ConvolutionRacah}
We have, for $l=0,1,\dots,N$,
\[w_l=\sum_{k=0}^n\binom{N}{l}\frac{(\lambda_2)_l(\lambda_1)_{N-l}(\lambda_1+\lambda_2+\lambda_3+N-1)_k}{(\lambda_1)_k(\lambda_1+\lambda_2+k-1)_k(\lambda_1+\lambda_2+2k)_{N-k}}R_{l,k}v_k\ ,\]
\[v_l=\sum_{k=0}^n\binom{N}{l}\frac{(\lambda_2)_l(\lambda_3)_{N-l}(\lambda_1+\lambda_2+\lambda_3+N-1)_k}{(\lambda_3)_k(\lambda_2+\lambda_3+k-1)_k(\lambda_2+\lambda_3+2k)_{N-k}}R_{k,l}w_k\ .\]
\end{theo}
\begin{proof} Recall that the representation space has become after the changes of variables, respectively, the space of polynomials in $t,tu_2,tu_1(1-u_2)$ and the space of polynomials in $t,tv_2,tv_1(1+v_2)$. From their explicit expressions \eqref{eq:vwtuv} we have that the polynomials $t^nv_l$, where $n,N$ vary in $\mathbb{N}$ and $l\in\{0,\dots,N\}$ form a basis of this space (and similarly for the vectors $t^nw_l$). 

Consider the two operators $C'_{123}$ and $t\partial_t$. They commute and the vectors $t^nv_l$ and $t^nw_l$ are common eigenvectors with eigenvalues, respectively, $N(\la_1+\la_2+\la_3+N-1)$ and $n+N$. We used Proposition \ref{prop:actionRacah} and the fact that $t$ is transparent to $C'_{123}$. Note that the eigenvalues $N(\la_1+\la_2+\la_3+N-1)$ are different for different values of $N$, this is where we used the conditions on $\la_1,\la_2,\la_3$ fixed from the beginning in \eqref{cond-Racah}. We conclude that both sets $\{v_l\}_{l=0,\dots,N}$ and $\{w_l\}_{l=0,\dots,N}$ span the same subspace (the common eigenspace of $C'_{123}$ and $t\partial_t$).

Now we are in position to apply the results of Proposition \ref{prop-Racah}, since $C'_{123}$ takes the correct value $N(\la_1+\la_2+\la_3+N-1)$ and $v_l$ and $w_l$ are eigenvectors of respectively $X$ and $Y$ with the correct eigenvalues. Therefore we have that each eigenvector $v_l$ of $X$ (respectively, each eigenvector $w_l$ of $Y$) must be proportional to the vector $\tilde{v}_l$ (respectively, $\tilde{w}_l$) obtained in Proposition \ref{prop-Racah}. We set  $v_l=\alpha_l\tilde{v}_l$ and $w_l=\beta_l\tilde{w}_l$. The equalities (\ref{transition-Racah}) involving the Racah coefficients now read:
\[w_l=\beta_l\sum_{k=0}^N\alpha_k^{-1}R_{l,k} v_k\ \ \ \text{and}\ \ \ v_l=\Gamma^{-1}\alpha_l \sum_{k=0}^N \beta_k^{-1}\tilde{R}_{k,l}w_k\,,\]
where $\Gamma$ was given explicitly. We take $y=0$ and then $x=0$ in the first equality, while we take $y=0$ and then $z=0$ in the second. Using the values of the Jacobi polynomials at $1$ et $-1$
\[P_k^{\alpha,\beta}(1)=\frac{(\alpha)_k}{k!}\ \ \ \ \text{and}\ \ \ \ \ P_k^{\alpha,\beta}(-1)=(-1)^kP_k^{\beta,\alpha}(1)=(-1)^k\frac{(\beta)_k}{k!}\ ,\] 
we get immediately:
\[\beta_l=\binom{N}{l}\frac{(\lambda_1)_{N-l}(\lambda_2)_{l}}{(\lambda_1+\lambda_2)_N}\alpha_0\ \ \ \text{and}\ \ \ \alpha_l=\binom{N}{l}\frac{(\lambda_3)_{N-l}(\lambda_2)_{l}}{(\lambda_2+\lambda_3)_N}\Gamma\tilde{R}_{0,l}^{-1}\beta_0\ .\]
In particular, we find $\alpha_0\beta_0^{-1}=\frac{(\lambda_3)_N}{(\lambda_2+\lambda_3)_N}\Gamma$. Combining both formulas above leads to:
\[\beta_l\alpha_k^{-1}=\binom{N}{l}\binom{N}{k}^{-1}\frac{(\lambda_1)_{N-l}(\lambda_2)_l(\lambda_3)_N}{(\lambda_2)_k(\lambda_3)_{N-k}(\lambda_1+\lambda_2)_N}\tilde{R}_{0,k}\ .\]
It remains to calculate $\tilde{R}_{0,k}$, which is by definition $\frac{B(0)\dots B(k-1)}{D(1)\dots D(k)}$. Using \eqref{ABCD-Racah}, we find:
\[\tilde{R}_{0,k}=\binom{N}{k}\frac{(\la_2)_k(\la_1+\la_2+\la_3+N-1)_{k}(\la_1+\la_2)_N}{(\la_1)_k(\la_3+N-k)_k(\la_1+\la_2+k-1)_k(\la_1+\la_2+2k)_{N-k}}\ .\]
Plugging this into the preceding formula, and doing some easy simplifications, we conclude that:
\[\beta_l\alpha_k^{-1}=\binom{N}{l}\frac{(\lambda_2)_l(\lambda_1)_{N-l}(\lambda_1+\lambda_2+\lambda_3+N-1)_k}{(\lambda_1)_k(\lambda_1+\lambda_2+k-1)_k(\lambda_1+\lambda_2+2k)_{N-k}}\ ,\]
thereby proving the first formula in the theorem. 

The second can by obtained by similar manipulations starting from the formula just obtained for $\beta_l\alpha_k^{-1}$. More efficiently, we note that $v_l$ and $w_l$ are exchanged by the operation $\lambda_1\leftrightarrow \lambda_3$ (up to renaming $x$ by $z$ and vice versa). This operation also exchanges $R_{k,l}$ and $R_{l,k}$. This proves immediately the inverse formula.
\end{proof}

\begin{rema}\label{rema-tensorproduct3}
Similarly to Remark \ref{rema-tensorproduct2}, there is a natural interpretation of the construction above in terms the representation theory of $sl_2$. For a given $N\geq 0$, the polynomials $t^{n}v_l$ (or $t^nw_l$), with $l=0,\dots,N$ and $n\geq 0$, form a basis of a summand in the tensor product $V_{\la_1}\otimes V_{\la_2}\otimes V_{\la_3}$ isomorphic to $N+1$ copies of $V_{\la_1+\la_2+\la_3+2N}$. In fact, this proves the existence of such a summand, under the restrictions \eqref{cond-Racah} on the parameters $\la_1,\la_2,\la_3$.
If these restrictions are satisfied for any positive integer $N$, we have realised explicitly the decomposition of the tensor product:
\[V_{\la_1}\otimes V_{\la_2}\otimes V_{\la_3}=\bigoplus_{N\geq 0} (N+1)V_{\la_1+\la_2+\la_3+2N}\ .\]
Note that the sets $\{v_l\}_{l=0,\dots,N}$ and $\{w_l\}_{l=0,\dots,N}$ are different bases for the space of lowest-weight vectors for the weight $\la_1+\la_2+\la_3+2N$, each one diagonalising, respectively, $C_{12}$ and $C_{23}$.
\end{rema}

\section{Higher-rank Racah algebras and Jacobi operators}

Let $n\geq 3$. The higher-rank Racah algebra (or more precisely, its special quotient following the terminology of \cite{CGPV22}) is the subalgebra of $U(sl_2)^{\otimes n}$ generated by the intermediate Casimir elements. Take any subset $I\subset\{1,\dots,n\}$ and define the associated embedding of $sl_2$ in the $n$-fold tensor product:
\[\delta_I(x)=\sum_{i\in I}1\otimes \dots 1 \otimes \stackrel{(i)}{x}\otimes 1\dots \otimes 1\ \ \ \ \ \text{for any $x\in sl_2$\ ,}\]
where in the sum above $x$ is in position $i$. This extends to an algebra embedding of $U(sl_2)$ into the $n$-fold tensor product and the intermediate Casimir element indexed by $I$ is the image of the Casimir element $C$:
\[C_I=\delta_I(C)\ \ \ \ \ \ \text{for any $I\subset\{1,\dots,n\}$.}\]
We will often drop the accolades in the notation for sets, and denote for example $C_{125}=C_{\{1,2,5\}}$.

Note that we have:
\begin{equation}\label{eq:CIcommut}
[C_I,C_J]=0\ \ \ \ \ \ \text{if $I\subset J$ or $I\cap J=\emptyset$\,.}
\end{equation}
This is obvious if $I\cap J=\emptyset$. If $I\subset J$, it is enough, up to permuting some components, to look at $I=\{1,\dots,k\}$ and $J=\{1,\dots,k+l\}$ in $U(sl_2)^{k+l}$. Then we have 
\[C_I=\delta_{\{1,\dots,k\}}(C\otimes 1^{\otimes^{l-1}})\ \ \ \ \ \ \text{and}\ \ \ \ \ \ \ C_J=\delta_{\{1,\dots,k\}}\bigl(\delta_{\{1,\dots,l\}}(C)\bigr)\ .\]
Since $C\otimes 1^{\otimes^{l-1}}$ and $\delta_{\{1,\dots,l\}}(C)$ commute in $U(sl_2)^{\otimes l}$ (recall that $C$ is central), their images $C_I$ and $C_J$ commute as well.

\paragraph{$n$-fold tensor product of Verma modules.} We fix complex parameters $\la_1,\dots,\la_n$ and we consider the tensor product $V_{\la_1}\otimes\dots\otimes V_{\la_n}$ of lowest-weight Verma modules. As in the previous sections, we use the realisation on polynomials. The representation space is now $Pol(x_1,\dots,x_n)$, the variable $x_i$ corresponding to the $i$-th factor of the tensor product.

In this section, for simplicity, we will identify elements $C_I$ with their action on $Pol(x_1,\dots,x_n)$. We note that the Casimir elements acting in a single factor are numbers:
\[C_i=\frac{\la_i(\la_i-2)}{4}\ \ \ \ \ \text{for any $i=1,\dots,n$.}\]
We shift as we did before the elements $C_I$:
\[C'_I=C_I-\frac{\la_I(\la_I-2)}{4}\ \ \ \ \ \ \ \text{where $\la_I=\sum_{i\in I}\la_i$}\,.\]
With these notations, we have $C'_1=\dots=C'_n=0$.

\subsection{Coupling schemes and commutative subalgebras}

We will discuss coupling schemes in terms of set partitions of $\{1,\dots,n\}$ in a way adapted to our notations $C_I$ and to the goal of describing the commutative subalgebras. Equivalent ways use binary trees or parenthesised words \cite{FLVJ02,LVJ02,VdJ03}. The correspondences between the different possibilities are easy to work out.

\paragraph{Coupling steps.} A coupling step will relate two partitions of the set $\{1,\dots,n\}$. We denote by $I_1\sqcup \dots \sqcup I_k$ a partition of $\{1,\dots,n\}$, where $I_1,\dots,I_k$ are pairwise disjoint and their union is $\{1,\dots,n\}$. We always assume that the appearing subsets are ordered such that $\text{min}(I_1)<\dots <\text{min}(I_k)$.

We define a coupling step to be a transformation of a partition of $\{1,\dots,n\}$ into another such partition as follows:
\[I_1\sqcup \dots \sqcup I_k\ \stackrel{(I_a,I_b)}{\longrightarrow}\ J_1\sqcup \dots \sqcup J_{k-1}\ ,\]
where the subsets $I_a$ and $I_b$ have been removed from the first partition and replaced by their union $I_a\cup I_b$. We see that a coupling step always reduces the number of subsets by one. We denote the coupling step by the two subsets involved: $(I_a,I_b)$. 

To each coupling step is associated a certain change of variables. We start from a set of variables indexed by the subsets of the first partition $(x_{I_1},\dots, x_{I_k})$ and we replace the two variables $x_{I_a},x_{I_b}$ by:
\begin{equation}\label{partialchange}
x_{I_a\cup I_b}=x_{I_a}+x_{I_b}\ \ \ \ \ \text{and}\ \ \ \ \ u_{I_a,I_b}=\frac{x_{I_b}-x_{I_a}}{x_{I_a}+x_{I_b}}\ .
\end{equation}

\paragraph{Coupling schemes and commutative subalgebras.} We define a coupling scheme $\Gamma$ to be a sequence of $n-1$ coupling steps from $\{1\}\sqcup \{2\}\sqcup\dots \sqcup \{n\}$ to the trivial partition $\{1,\dots,n\}$:
\begin{equation}\label{Gamma}
\Gamma\ :\ \{1\}\sqcup \dots \sqcup \{n\}\stackrel{(I_1,J_1)}{\longrightarrow} \dots \dots \dots \stackrel{(I_{n-1},J_{n-1})}{\longrightarrow} \{1,\dots,n\}\ .
\end{equation}
For each step denoted $(I_k,J_k)$, we associate the element $C'_{I_k\cup J_k}$ and we put all these elements together to form the set:
\[S_{\Gamma}=\{C'_{I_1\cup J_1},\dots ,C'_{I_{n-1}\cup J_{n-1}}\}\ .\]
It is quite obvious that the set $S_{\Gamma}$ is commutative. Indeed, after the first step the subset $I_1\cup J_1$ is created in the partition. Then either the second step does not involve this subset, and then $I_2\cup J_2$ will be disjoint from $I_1\cup J_1$, or it involves it, and then $I_2\cup J_2$ will contain $I_1\cup J_1$. In both cases, $C'_{I_1\cup J_1}$ and $C'_{I_2\cup J_2}$ commute thanks to \eqref{eq:CIcommut}. A similar reasoning applies for the whole set $S_{\Gamma}$. 

Note that the last element in $S_{\Gamma}$ is always $C'_{1\dots n}$\,, while the first element always has only two indices $C'_{ij}$ with $i,j\in\{1,\dots,n\}$.

\paragraph{Change of variables.} Associated to a coupling scheme $\Gamma$ as in \eqref{Gamma}, there is a change of variables, which is obtained by applying successively the partial change of variables in \eqref{partialchange} starting from the set of variables $x_1,\dots,x_n$. This results in:
\begin{equation}\label{fullchange}
(x_1,\dots,x_n)\mapsto (x_{1\dots n}\,,\,u_{I_1,J_1}\,,\,\dots\,,\, u_{I_{n-1},J_{n-1}}\,)\ ,
\end{equation}
where $x_{1\dots n}=x_1+\dots+x_n$ and more generally, the notations are, for $I,J\subset\{1,\dots,n\}$:
\[x_I=\sum_{i\in I}x_i\ \ \ \ \ \text{and}\ \ \ \ \ \ u_{I,J}=\frac{x_{J}-x_{I}}{x_{I}+x_{J}}\ .\] 
Each coupling step creates out of two variables $x$ a variable $u$ and a new variable $x$. Consider the $k$-th step corresponding to the subsets $(I_k,J_k)$. Just before it, we have already created the $k-1$ variables $u_{I_1,J_1},\dots,u_{I_{k-1},J_{k-1}}$ which will be fixed until the end. And we have at hand variables $x$ corresponding to the subsets in the partition reached after $k-1$ steps. In particular, we have the variables $x_{I_k}$ and $x_{J_k}$. The $k$-th step uses these two variables to create the new variables $x_{I_k\cup J_k}$ and $u_{I_k,J_k}$.

\begin{exam}\label{exam-coupling}
For $n=3$, to the coupling scheme $1\sqcup 2\sqcup 3 \to 12\sqcup 3 \to 123$ are associated the following subalgebra and new variables:
\[S_{\Gamma}=\{C'_{12},C'_{123}\}\ \ \ \ \ \text{and}\ \ \ \ \ (x_1+x_2+x_3,\frac{x_2-x_1}{x_1+x_2},\frac{x_3-(x_1+x_2)}{x_1+x_2+x_3})\ .\]
We recover the change of variables \eqref{tv-Racah1} from the previous section. The other change of variables \eqref{tv-Racah2} is recovered with the coupling scheme  $1\sqcup 2\sqcup 3 \to 1\sqcup 23 \to 123$, which corresponds to the subalgebra $\{C'_{23},C'_{123}\}$. Note that we have also a third coupling scheme resulting in the subalgebra $\{C'_{13},C'_{123}\}$.

An example for $n=4$ is $1\sqcup 2\sqcup 3\sqcup 4 \to 12\sqcup 3\sqcup 4 \to 12\sqcup 34\to 1234$, resulting in the following subalgebra and new variables:
\[S_{\Gamma}=\{C'_{12},C'_{34},C'_{1234}\}\ \ \ \ \text{and}\ \ \ \ (x_1+x_2+x_3+x_4,\frac{x_2-x_1}{x_1+x_2},\frac{x_4-x_3}{x_3+x_4},\frac{x_3+x_4-(x_1+x_2)}{x_1+x_2+x_3+x_4})\ .\]
\end{exam}

\begin{rema}
It is easy to see that different coupling schemes can produce the same commutative subalgebra. Consider for example $1\sqcup 2\sqcup 3\sqcup 4 \to 1\sqcup 2\sqcup 34 \to 12\sqcup 34\to 1234$ in comparison with the previous example. The number of coupling schemes is easily seen to be $\prod_{k=3}^n\binom{k}{2}=\frac{n!(n-1)!}{2^{n-1}}$, while the number of obtained commutative subalgebras is $(2n-3)!!=3\times 5\times\dots \times 2n-3$, see \cite{FLVJ02}.
\end{rema}

\subsection{Realisations with Jacobi operators}

We fix a coupling scheme $\Gamma$:
\begin{equation}\label{Gamma2}
\Gamma\ :\ \{1\}\sqcup \dots \sqcup \{n\}\stackrel{(I_1,J_1)}{\longrightarrow} \dots \dots \dots \stackrel{(I_{n-1},J_{n-1})}{\longrightarrow} \{1,\dots,n\}\ .
\end{equation}
We have an associated commutative set $S_{\Gamma}=\{C'_{I_1\cup J_1},\dots ,C'_{I_{n-1}\cup J_{n-1}}\}$ and a new set of variables obtained from $x_1,\dots,x_n$:
\[(x_{1\dots n}\,,\,u_{I_1,J_1}\,,\,\dots\,,\, u_{I_{n-1},J_{n-1}}\,)\ ,\]
given in \eqref{fullchange}. Now our goal is to show how the operators in the set $S_{\Gamma}$ are expressed using Jacobi operators and the $n-1$ new variables $u$.
\begin{prop}\label{prop:Cgeneraln}
We have:
\[C'_{I_k\cup J_k}=\Phi^{\la_{I_k},\la_{J_k}}_{u_{I_k,J_k}}+\frac{2}{1-u_{I_k,J_k}}C'_{I_k}+\frac{2}{1+u_{I_k,J_k}}C'_{J_k}\ .\]
\end{prop}
Before giving the proof, we note that $C'_{I_k}$ is either $0$, if $|I_k|=1$, or already inside $S_{\Gamma}$ (and similarly for $C'_{J_k}$). Therefore, the proposition indeed allows, recursively on $k$, to express any element of $S_{\Gamma}$ using the new variables $u$. Examples will follow.
\begin{proof}
We prove the formula by recursion on $k=1,\dots,n-1$, by performing the change of variables step by step, each step being given in \eqref{partialchange}. Just before the $k$-th step in the change of variables, we have at hand a set of variables:
\[\vec{u}=(u_{I_1,J_1},\dots,u_{I_{k-1},J_{k-1}})\ \ \ \ \text{and}\ \ \ \ \ \vec{x}=(\dots,x_{I_k},\dots,x_{J_k},\dots)\ ,\]
where the variables in $\vec{x}$ are indexed by the subsets in the partition of $\{1,\dots,n\}$ reached after $k-1$ steps. The variables in $\vec{u}$ will not be modified anymore. The $k$-step removes the two variables $x_{I_k}$ and $x_{J_k}$ and creates a new variable in $\vec{u}$ and a new variable in $\vec{x}$ by:
\begin{equation}\label{kstep}
x_{I_k\cup J_k}=x_{I_k}+x_{J_k}\ \ \ \ \ \text{and}\ \ \ \ \ u_{I_k,J_k}=\frac{x_{J_k}-x_{I_k}}{x_{I_k}+x_{J_k}}\ .
\end{equation}
We need to calculate the images by $\delta_{I_k\cup J_k}$ of the generators of $sl_2$ and in turn of the Casimir element, and we claim that after the $k$-th step in the change of variables, we have:
\begin{equation}\label{claimedformulas}
\begin{array}{lcl}
\delta_{I_k\cup J_k}(H) & \mapsto & \la_{I_k\cup J_k}+2x_{I_k\cup J_k}\partial_{x_{I_k\cup J_k}}\ ,\\[0.2em]
\delta_{I_k\cup J_k}(E) & \mapsto & x_{I_k\cup J_k}\ ,\\[0.2em]
\delta_{I_k\cup J_k}(F) & \mapsto  & -x_{I_k\cup J_k}\partial_{x_{I_k\cup J_k}}^2-\la_{I_k\cup J_k}\partial_{x_{I_k\cup J_k}}+x_{I_k\cup J_k}^{-1}\widetilde{F}_{I_k\cup J_k}\ ,
\end{array}
\end{equation}
where $\widetilde{F}_{I_k\cup J_k}$ is given recursively by
\[\widetilde{F}_{I_k\cup J_k}=\Phi^{\la_{I_k},\la_{J_k}}_{u_{I_k,J_k}}+\frac{2}{1-u_{I_k,J_k}}\widetilde{F}_{I_k}+\frac{2}{1+u_{I_k,J_k}}\widetilde{F}_{J_k}\ \ \ \ \text{and}\ \ \ \ \widetilde{F}_i=0\,,\ i=1,\dots,n\ .\ \ \ \]
From this, calculating $C'_{I_k\cup J_k}$ is immediate. The part without $\widetilde{F}_{I_k\cup J_k}$ is the usual one-variable calculation which gives exactly the constant that we substracted from $\delta_{I_k\cup J_k}(C)$ to define $C'_{I_k\cup J_k}$. The remaining term gives that $C'_{I_k\cup J_k}=\widetilde{F}_{I_k\cup J_k}$. Note that this only involves the variables $u$ and therefore will remain unmodified until the end of the change of variables. The proposition then follows.

\vskip .1cm
So we are left with proving \eqref{claimedformulas}. We note that $\delta_{I_k\cup J_k}(X)=\delta_{I_k}(X)+\delta_{J_k}(X)$ for any $X\in\{H,E,F\}$, and thus just before the $k$-th step in the change of variables, we have:
\[
\begin{array}{lcl}
\delta_{I_k\cup J_k}(H) & \mapsto & \la_{I_k}+2x_{I_k}\partial_{x_{I_k}}+\la_{J_k}+2x_{J_k}\partial_{x_{J_k}}\ ,\\[0.2em]
\delta_{I_k\cup J_k}(E) & \mapsto & x_{I_k}+x_{J_k}\ ,\\[0.2em]
\delta_{I_k\cup J_k}(F) & \mapsto  & -x_{I_k}\partial_{x_{I_k}}^2-\la_{I_k}\partial_{x_{I_k}}-x_{J_k}\partial_{x_{J_k}}^2-\la_{J_k}\partial_{x_{J_k}}+x_{I_k}^{-1}\widetilde{F}_{I_k}+x_{J_k}^{-1}\widetilde{F}_{J_k}\ .
\end{array}
\]
This is obviously true if $k=1$ since in this case $|I_k|=|J_k|=1$ and $\widetilde{F}_{I_k}=\widetilde{F}_{J_k}=0$, and we have used the recurrence hypothesis if $k>1$.

Now we perform the $k$-th step \eqref{kstep} in the change of variables. The calculation with the terms not involving $\widetilde{F}_{I_k}$ and $\widetilde{F}_{J_k}$ is a two-variable calculation already done in Proposition \ref{propmodele(t,v)}. This gives at once the desired formulas \eqref{claimedformulas} for $H$ and $E$, while for $F$, we have at this point:
\[\delta_{I_k\cup J_k}(F)\ \mapsto\  -x_{I_k\cup J_k}\partial_{x_{I_k\cup J_k}}^2-\la_{I_k\cup J_k}\partial_{x_{I_k\cup J_k}}+x_{I_k\cup J_k}^{-1}\Phi^{\la_{I_k},\la_{J_k}}_{u_{I_k,J_k}}+x_{I_k}^{-1}\widetilde{F}_{I_k}+x_{J_k}^{-1}\widetilde{F}_{J_k}\ .\]
By recurrence, $\widetilde{F}_{I_k}$ and $\widetilde{F}_{J_k}$ only involve the variables $u$ so we leave them as they are, and we only need to transform the variables $x_{I_k},x_{J_k}$ into the new variables. The formulas inverses to \eqref{kstep} are:
\[x_{I_k}=\frac{x_{I_k\cup J_k}(1-u_{I_k,J_k})}{2}\ \ \ \ \ \text{and}\ \ \ \ \ x_{J_k}=\frac{x_{I_k\cup J_k}(1+u_{I_k,J_k})}{2}\ .\]
This concludes the verification of \eqref{claimedformulas} and the proof of the proposition.
\end{proof}
\begin{exam}
For $n=3$, we recover in particular Proposition \ref{prop:actionRacah}. For $n=4$, consider the coupling scheme giving the set $\{C'_{12},C'_{123},C'_{1234}\}$. The new variables are:
\[(x_{1234}\,,\,u_{1,2}\,,\, u_{12,3}\,,\, u_{123,4})=(x_1+x_2+x_3+x_4\,,\, \frac{x_2-x_1}{x_1+x_2}\,,\, \frac{x_3-(x_1+x_2)}{x_1+x_2+x_3}\,,\, \frac{x_4-(x_1+x_2+x_3)}{x_1+x_2+x_3+x_4})\]
and the commutative set is realised by $C'_{12}=\Phi_{u_{1,2}}^{\la_1,\la_2}$, $C'_{123}=\Phi_{u_{12,3}}^{\la_1+\la_2,\la_3}+\frac{2}{1-u_{12,3}}\Phi_{u_{1,2}}^{\la_1,\la_2}$ and
\[C'_{1234}=\Phi_{u_{123,4}}^{\la_1+\la_2+\la_3,\la_4}+\frac{2}{1-u_{123,4}}\Phi_{u_{12,3}}^{\la_1+\la_2,\la_3}+\frac{4}{(1-u_{123,4})(1-u_{12,3)}}\Phi_{u_{1,2}}^{\la_1,\la_2}\ .\]
For the coupling scheme of Example \ref{exam-coupling} giving the set $\{C'_{12},C'_{34},C'_{1234}\}$, the new variables are denoted $u_{1,2}$, $u_{3,4}$ and $u_{12,34}$ and the operators are:
\[C'_{12}=\Phi_{u_{1,2}}^{\la_1,\la_2}\,,\ \ \ C'_{34}=\Phi_{u_{3,4}}^{\la_3,\la_4}\,,\ \ \ 
C'_{1234}=\Phi_{u_{12,34}}^{\la_1+\la_2,\la_3+\la_4}+\frac{2}{1-u_{12,34}}\Phi_{u_{1,2}}^{\la_1,\la_2}+\frac{2}{1+u_{12,34}}\Phi_{u_{3,4}}^{\la_3,\la_4}\ .\]
\end{exam}

\subsection{Eigenvectors and Jacobi polynomials}

We need some further notation. The coupling scheme $\Gamma$ is still fixed. For $1\leq a<b\leq n-1$, we denote:
\[\left\{\begin{array}{l} a\prec_L b\ \ \ \quad\text{if $I_a\cup J_a\subset I_b$\,,}\\[0.4em]
a\prec_R b\ \ \ \quad\text{if $I_a\cup J_a\subset J_b$\,.}\end{array}\right.\]
In words, for a fixed $b$, the indices $a$ such that $a\prec_L b$ are the indices of the steps in the coupling scheme which were used to produce the subset $I_b$ (similarly for $a\prec_Rb$ and $J_b$). We also denote:
\[a\preceq b\ \ \ \quad \text{if}\ \ a=b\ \ \text{or}\ \ a\prec_Lb\ \ \text{or}\ \ a\prec_R b\,.\] 
\begin{exam}
Consider the scheme $1\sqcup 2\sqcup3\sqcup4\stackrel{(1)}{\longrightarrow}12\sqcup 3\sqcup 4\stackrel{(2)}{\longrightarrow} 123\sqcup 4\stackrel{(3)}{\longrightarrow} 1234$, where the number of the steps are shown explicitly. We have $1\prec_L 2$ while no $a$ satisfies $a\prec_R 2$ (since the set $\{3\}$ used on the right in step 2 is of cardinal 1). We have $1\prec_L 3$ and $2\prec_L 3$ while no $a$ satisfies $a\prec_R 3$.
\end{exam}
Finally we define the following shift of the parameters $\lambda$'s, depending on a vector of integers $\vec{k}=(k_1,\dots,k_{n-1})$:
\[(\la^{(\vec{k})}_{I_b},\la^{(\vec{k})}_{J_b})=(\la_{I_b}+2\sum_{a\prec_L b}k_a\,,\,\la_{J_b}+2\sum_{a\prec_R b}k_a)\ ,\]
and we are ready to introduce the polynomials which will be eigenvectors of our operators:
\begin{equation}\label{vx_generaln}
v_{k_{1},\dots,k_{n-1}}=\prod_{b=1}^{n-1}(x_{I_b\cup J_b})^{k_b}P_{k_b}^{\la^{(\vec{k})}_{I_b},\la^{(\vec{k})}_{J_b}}(u_{I_b,J_b})\ ,\ \ \ \ \ \ k_1,\dots,k_{n-1}\geq 0\ .
\end{equation}
For $n=3$, depending on the coupling scheme, we recover the vectors $v_l$ and $w_l$ from \eqref{eq:vwxyz} (when we set $k_1=l$ and $N=k_1+k_2$). The defining formula above allows to express easily the polynomials in terms of the variables $x_1,\dots,x_{n-1}$. We refer to the examples given at the end. However, the polynomials can also be expressed entirely in terms of the new variables $(x_{1\dots n},u_{I_1,J_1},\dots,u_{I_{n-1},J_{n-1}})$ associated to the scheme. The formula will appear in the proof below in \eqref{vu_generalN}. 

The following proposition generalises Proposition \ref{prop:eigenvectorsRacah3} (the case $n=3$) by providing the eigenvectors of the commutative family associated to any coupling scheme for any $n$. 
\begin{prop}
The vectors $\{v_{k_1,\dots,k_{n-1}}\}_{k_1,\dots,k_{n-1}\geq 0}$ are common eigenvectors for the commutative family $S_{\Gamma}$ of operators, with eigenvalues given by:
\[C'_{I_b\cup J_b}(v_{k_1,\dots,k_{n-1}})=\bigl(\sum_{a\preceq b}k_a\bigr)\bigl(\sum_{a\preceq b}k_a+\la_{I_b}+\la_{J_b}-1\bigr)v_{k_1,\dots,k_{n-1}}\ .\]
\end{prop}
\begin{proof}
First we need to express the polynomials $v_{k_1,\dots,k_{n-1}}$ only in terms of the variables $x_{1\dots n}$ and the $u$'s. We have:
\begin{equation}\label{vu_generalN}
v_{k_{1},\dots,k_{n-1}}=\left(\frac{x_{1\dots n}}{2}\right)^{k_1+\dots+k_{n-1}}\prod_{b=1}^{n-1}(1-u_{I_b,J_b})^{\sum_{a\prec_Lb}k_a}(1+u_{I_b,J_b})^{\sum_{a\prec_Rb}k_a}P_{k_b}^{\la^{(\vec{k})}_{I_b},\la^{(\vec{k})}_{J_b}}(u_{I_b,J_b})\ .
\end{equation}
To see this, let $a\in\{1,\dots,n-1\}$ and consider the factor $(x_{I_a\cup J_a})^{k_a}$ appearing in \eqref{vx_generaln}. We need to rewrite it in terms of the new variables (only if $a<n-1$ since otherwise it is equal to $x_{1\dots n}^{k_{n-1}}$). Now let $b>a$ be the index such that $I_a\cup J_a$ is either equal to $I_b$ or $J_b$. Then the $b$-th step in the change of variable gives that 
\[2x_{I_a\cup J_a}=\left\{\begin{array}{l} x_{I_b\cup J_b}(1-u_{I_b,J_b})\ \ \ \quad\text{if $I_a\cup J_a= I_b$\,,}\\[0.4em]
x_{I_b\cup J_b}(1+u_{I_b,J_b})\ \ \ \quad\text{if $I_a\cup J_a= J_b$\,.}\end{array}\right.\]
Reproducing this for $x_{I_b\cup J_b}$ until we reach $x_{1\dots n}$, we see that the factor $(x_{I_a\cup J_a})^{k_a}$ contributes first to an overall factor $\left(\frac{x_{1\dots n}}{2}\right)^{k_a}$, and second to the factors $(1\pm u_{I_b,J_b})$ for any $b$ such that either $a\prec_Lb$ or $a\prec_R b$, the sign being chosen accordingly. This proves formula \eqref{vu_generalN}.

Now we can apply the operators $C'_{I_b\cup J_b}$, using the recursive formula in Proposition \ref{prop:Cgeneraln}, on the vectors $v_{k_1,\dots,k_{n-1}}$ given by \eqref{vu_generalN}. We use induction on $b$. For $b=1$, we have $C'_{I_1\cup J_1}=\Phi^{\la_{I_1},\la_{J_1}}_{u_{I_1,J_1}}$. Since only the variable $u_{I_1,J_1}$ appears, the application on $v_{k_1,\dots,k_{n-1}}$ gives immediately the eigenvalue $k_1(k_1+\la_{I_1}+\la_{J_1}-1)$ and the proposition is proved for $b=1$.

Then let $b>1$ and use the recursive formula:
\[C'_{I_b\cup J_b}=\Phi^{\la_{I_b},\la_{J_b}}_{u_{I_b,J_b}}+\frac{2}{1-u_{I_b,J_b}}C'_{I_b}+\frac{2}{1+u_{I_b,J_b}}C'_{J_b}\ .\]
By the recurrence hypothesis, the last two terms applied on $v_{k_1,\dots,k_{n-1}}$ give the following coefficient:
\begin{equation}\label{action1}
\frac{2}{1-u_{I_b,J_b}}\sum_{a\prec_L b}k_a\bigl(\sum_{a\prec_L b}k_a+\la_{I_b}-1\bigr)+\frac{2}{1+u_{I_b,J_b}}\sum_{a\prec_R b}k_a\bigl(\sum_{a\prec_R b}k_a+\la_{J_b}-1\bigr)\ .
\end{equation}
It remains to apply $\Phi^{\la_{I_b},\la_{J_b}}_{u_{I_b,J_b}}$. One needs to apply it only on the factor containing the variable $u_{I_b,J_b}$ in the product \eqref{vu_generalN}. From Lemma \ref{lemm:Jacobioperator} used in the preceding section together with $\Phi_{-u}^{\beta,\alpha}=\Phi_{u}^{\alpha,\beta}$, one finds:
\begin{multline*}
\Phi_u^{\alpha,\beta}\circ (1-u)^l(1+u)^m=(1-u)^l(1+u)^m\circ \Bigl(\Phi_u^{\alpha+2l,\beta+2m}\\
+l\beta+m\alpha+2lm-l(l+\alpha-1)\frac{1+u}{1-u}-m(m+\beta-1)\frac{1-u}{1+u}\Bigr)\ .
\end{multline*}
This shows that $\Phi^{\la_{I_b},\la_{J_b}}_{u_{I_b,J_b}}$ applied on $v_{k_1,\dots,k_{n-1}}$ give the following coefficient:
\begin{multline*}
k_b(k_b+\la_{I_b}+\la_{J_b}+2\sum_{a\prec_L b}k_a+2\sum_{a\prec_R b}k_a -1)+\la_{J_b}\sum_{a\prec_L b}k_a+\la_{I_b}\sum_{a\prec_R b}k_a+2\sum_{a\prec_L b}k_a\sum_{a\prec_R b}k_a\\
-\sum_{a\prec_L b}k_a(\sum_{a\prec_L b}k_a+\la_{I_b}-1)\frac{1+u_{I_b,J_b}}{1-u_{I_b,J_b}}-\sum_{a\prec_R b}k_a(\sum_{a\prec_R b}k_a+\la_{J_b}-1)\frac{1-u_{I_b,J_b}}{1+u_{I_b,J_b}}\ .
\end{multline*}
Summing with \eqref{action1}, this combines to produce the eigenvalue as claimed in the proposition.
\end{proof}

\paragraph{Examples.} We conclude this section with some examples of polynomials associated to coupling schemes for $n=4$. We use the notations $\la_I=\sum_{i\in I}\la_i$, $x_I=\sum_{i\in I}x_i$ and $u_{I,J}=\frac{x_J-x_I}{x_{I}+x_J}$.\\
$\bullet$ For the scheme $1\sqcup 2\sqcup3\sqcup4\rightarrow 12\sqcup 3\sqcup 4\rightarrow 123\sqcup 4\rightarrow 1234$ giving the commutative set $\{C'_{12},C'_{123},C'_{1234}\}$, the vectors are:
\[(x_{1234})^{k_3}(x_{123})^{k_2}(x_{12})^{k_1}P_{k_3}^{\la_{123}+2k_1+2k_2,\la_4}(u_{123,4})P_{k_2}^{\la_{12}+2k_1,\la_3}(u_{12,3})P_{k_1}^{\la_1,\la_2}(u_{1,2})\ .\]
$\bullet$ For the scheme $1\sqcup 2\sqcup3\sqcup4\rightarrow 12\sqcup 3\sqcup 4\rightarrow 12\sqcup 34\rightarrow 1234$ giving the commutative set $\{C'_{12},C'_{34},C'_{1234}\}$, the vectors are:
\[(x_{1234})^{k_3}(x_{34})^{k_2}(x_{12})^{k_1}P_{k_3}^{\la_{12}+2k_1,\la_{34}+2k_2}(u_{12,34})P_{k_2}^{\la_{3},\la_4}(u_{3,4})P_{k_1}^{\la_1,\la_2}(u_{1,2})\ .\]
$\bullet$ For the scheme $1\sqcup 2\sqcup3\sqcup4\rightarrow 1\sqcup 2\sqcup 34\rightarrow 1\sqcup 234\rightarrow 1234$ giving the commutative set $\{C'_{34},C'_{234},C'_{1234}\}$, the vectors are:
\[(x_{1234})^{k_3}(x_{234})^{k_2}(x_{34})^{k_1}P_{k_3}^{\la_{1},\la_{234}+2k_1+2k_2}(u_{1,234})P_{k_2}^{\la_{2},\la_{34}+2k_1}(u_{2,34})P_{k_1}^{\la_3,\la_4}(u_{3,4})\ .\]
$\bullet$ For the scheme $1\sqcup 2\sqcup3\sqcup4\rightarrow 1\sqcup 24\sqcup 3\rightarrow 1\sqcup 234\rightarrow 1234$ giving the commutative set $\{C'_{24},C'_{234},C'_{1234}\}$, the vectors are:
\[(x_{1234})^{k_3}(x_{234})^{k_2}(x_{24})^{k_1}P_{k_3}^{\la_{1},\la_{234}+2k_1+2k_2}(u_{1,234})P_{k_2}^{\la_{24}+2k_1,\la_3}(u_{24,3})P_{k_1}^{\la_2,\la_4}(u_{2,4})\ .\]
Using the convolution formula obtained for $n=3$ in Theorem \ref{theo:ConvolutionRacah}, under suitable conditions on parameters $\lambda_i$, one sees that each family of vectors is related to the next one with transition coefficients involving one Racah polynomial. Thus to pass from the first one to the last, the transition coefficients involve products of three Racah polynomials. These are examples of bivariate Racah polynomials. 

We refer to \cite{CFR23} where the graph of commutative subets $S_{\Gamma}$ is studied in details for $n=4$ along with the corresponding bivariate Racah polynomials appearing as transition coefficients. In the example above, the first subset and the last subset are at distance 3 in this graph, which explains that we need products of three Racah coefficients for the associated transition coefficients. For $n=4$, the diameter of the graph is 3. We refer to \cite{FLVJ02} for a study of the graph for general $n$.

\end{document}